\documentclass[reqno,12pt,a4paper]{amsart}

\usepackage{tikz}
\usetikzlibrary{matrix,arrows,calc,snakes,patterns,decorations.markings}

\usepackage[mathscr]{eucal}
\usepackage{graphics,epic}
\usepackage{amsfonts}
\usepackage{amscd}
\usepackage{latexsym}
\usepackage{amsmath,amssymb, amsthm, stmaryrd, bm, bbm}
\usepackage[all,2cell]{xy}
\usepackage{mathrsfs}

\newtheorem{theorem}{Theorem}[section]
\newtheorem*{theorem*}{Theorem}

\newtheorem{lemma}[theorem]{Lemma}
\newtheorem{proposition}[theorem]{Proposition}
\newtheorem{corollary}[theorem]{Corollary}

\newtheorem*{conjecture*}{Conjecture}

\newtheorem*{question*}{Question}
\theoremstyle{remark}
\newtheorem{remark}[theorem]{Remark}
\newtheorem{example}[theorem]{Example}
\theoremstyle{definition}

\newcommand{\MCM}{\opname{MCM}\nolimits}

\newcommand{\ie}{{\em i.e.~}\ }

\newcommand{\eg}{{\em e.g.~}\ }

\newcommand{\ul}[1]{\underline{#1}}

\newcommand{\opname}[1]{\operatorname{\mathsf{#1}}}

\renewcommand{\mod}{\opname{mod}\nolimits}

\newcommand{\proj}{\opname{proj}\nolimits}

\newcommand{\Mod}{\opname{Mod}\nolimits}

\newcommand{\Grmod}{\opname{Grmod}\nolimits}
\newcommand{\add}{\opname{add}\nolimits}

\renewcommand{\Im}{\opname{Im}\nolimits}

\newcommand{\rad}{\opname{rad}\nolimits}
\newcommand{\con}{\opname{cone}\nolimits}

\newcommand{\grproj}{\opname{grproj}}

\newcommand{\thick}{\opname{thick}\nolimits}
\newcommand{\Tria}{\opname{Tria}\nolimits}
\newcommand{\Loc}{\opname{Loc}\nolimits}
\newcommand{\per}{\opname{per}\nolimits}

\newcommand{\Z}{\mathbb{Z}}

\renewcommand{\P}{\mathbb{P}}

\newcommand{\ra}{\rightarrow}

\newcommand{\id}{\mathrm{id}}

\newcommand{\Hom}{\opname{Hom}}
\newcommand{\End}{\opname{End}}

\newcommand{\RHom}{\mathbb{R}\hspace{-2pt}\opname{Hom}}

\newcommand{\cHom}{\mathcal{H}\it{om}}
\newcommand{\cEnd}{\mathcal{E}\it{nd}}

\newcommand{\Ext}{\opname{Ext}}

\newcommand{\ten}{\otimes}
\newcommand{\lten}{\overset{\mathbb{L}}{\ten}}

\newcommand{\Tot}{\opname{Tot}}

\newcommand{\ca}{{\mathcal A}}

\newcommand{\cc}{{\mathcal C}}
\newcommand{\cd}{{\mathcal D}}
\newcommand{\ce}{{\mathcal E}}

\newcommand{\ch}{{\mathcal H}}

\newcommand{\cp}{{\mathcal P}}

\newcommand{\cs}{{\mathcal S}}
\newcommand{\ct}{{\mathcal T}}

\newcommand{\m}{\mathfrak{m}}
\newcommand{\del}{\partial}

\numberwithin{equation}{section}

\begin{document}

\title[Relative singularity categories II]{Relative singularity categories II: \\[1ex] DG models}

\author{Martin Kalck}
\thanks{M.K. was supported by the DFG grant Bu--1866/2--1 and ERSRC grant EP/L017962/1.}
\address{~~
}

\email{martin.maths@posteo.de}
\author{Dong Yang}
\thanks{D.Y. was supported by the National Science Foundation in China No. 11401297 and the DFG program SPP 1388 (YA297/1-1 and KO1281/9-1).}
\address{Dong Yang, Department of Mathematics, Nanjing University, Nanjing 210093, P. R. China}
\email{yangdong@nju.edu.cn}

\date{Last modified on \today}
\begin{abstract}
We study the relationship between singularity categories and relative singularity categories and discuss constructions of differential graded algebras of relative singularity categories. As consequences, we obtain structural results, which are known or generalise known results, on singularity categories of algebras with radical square zero, of non-commutative deformations of Kleinian singularities, of $SL_3(\mathbb{C})$-quotient singularities and of Gorenstein toric threefolds.\\
{\bf Keywords:} relative singularity category, singularity category, dg algebra, cluster category, quotient singularity\\
{\bf MSC 2010:} 14B05, 16E45, 18E30 
\end{abstract}

\maketitle

\tableofcontents

\section{Introduction}

This is the second paper of a series of papers studying the relationship between relative singularity categories and singularity categories.

\smallskip
Let $k$ be a commutative ring and let $A$ be a $k$-algebra which is right noetherian as a ring. Let $e\in A$ be an idempotent and let $R=eAe$. We define the \emph{singulaity category of $A$ relative to $e$} as the triangle quotient
\[
\Delta_e(A)=\ch^b(\proj A)/\thick(eA),
\]
where $\ch^b(\proj A)$ is the bounded homotopy category of finitely generated projective $A$-modules. The \emph{singularity category} of $R$ is defined as the triangle quotient
\[
\cd_{sg}(R)=\cd^b(\mod R)/\ch^b(\proj R),
\]
where $\cd^b(\mod R)$ is the bounded derived category of finitely generated $R$-modules. It turns out that $\cd_{sg}(R)$ is a canonical triangle quotient of $\Delta_e(A)$  (Corollary~\ref{cor:relative-singularity-category-vs-singularity-category}).
rhb
\begin{theorem}\label{thm:main-compare}
Assume that $A$ is flat over $k$ and that $A$ has finite global dimension. Then there is a triangle equivalence up to direct summands
\[
\cd_{sg}(R)\simeq\Delta_e(A)/\Delta_e(A)_{rhb},
\]
where $\Delta_e(A)_{rhb}$ is the full subcategory of $\Delta_e(A)$ consisting of objects $X$ such that for any $Y\in\Delta_e(A)$ the space $\Hom(Y,\Sigma^p X)$ vanishes for almost all $p\in\mathbb{Z}$.
\end{theorem}

The category $\Delta_e(A)$ is triangle equivalent to the perfect derived category $\per(B)$ of differential graded modules over some nice differential graded $k$-algebra $B$ (Corollary~\ref{c:restriction-and-induction}). In terms of $B$, Theorem~\ref{thm:main-compare} has the following presentation (Theorem~\ref{t:singularity-category-vs-relative-singularity-category}). Notice that this description of $\cd_{sg}(R)$ in terms of $B$ resembles the construction due to Amiot~\cite{Amiot09} of the cluster category of a quiver with potential from the associated Ginzburg differential graded algebra.

\begin{theorem}\label{thm:main-compare-2}
Assume that $A$ is flat over $k$ and that $A$ has finite global dimension. Then there is a triangle equivalence up to direct summands
\[
\cd_{sg}(R)\simeq\per(B)/\cd_{fg}(B),
\]
where $\cd_{fg}(B)$ is the derived category of differential graded $B$-modules whose total cohomology is finitely generated over the $0$-th cohomology of $B$.
\end{theorem}

 We then provide a description of $B$ as a dg quiver algebra in the case when $A$ is quasi-isomorphic to a dg quiver algebra (Lemma~\ref{lem:non-complete-cofibrant-model} and Theorem~\ref{t:construction-with-noncomplete-cofibrant-model}). 
\begin{theorem} \label{thm:main-construction}

Assume that $k$ is a field and assume that there is a quasi-isomorphism $\rho \colon  \tilde{A}\rightarrow A$ of dg algebras, where $\tilde{A}=(kQ,d)$ is a dg quiver algebra, such that $e$ is the image under $\rho$
of the sum $\tilde{e}$ of trivial paths at a subset $I$ of vertices of $Q$. Then $B$ is quasi-equivalent to the dg quiver algebra $(kQ',d')$, where $Q'$ is the quiver obtained from $Q$ by removing the vertices in $I$ and $d'$ is obtained from $d$ by removing the summands involving all paths passing through vertices in $I$.
\end{theorem}

We remark that both categories $\per$ and $\cd_{fg}$ are preserved up to triangle equivalence if  $B$ is replaced by a quasi-equivalent differential graded algebra. 
We will also recall two general constructions of $B$ due to Nicol\'as--Saor\'in \cite{NicolasSaorin09} and due to Drinfeld~\cite{Drinfeld04}. 

\smallskip
In the final part of this paper, we collect some examples for which Theorems~\ref{thm:main-compare-2} and~\ref{thm:main-construction} apply to give structural results on $\cd_{sg}(R)$, which are known or generalise known results.  These results are listed below.

\smallskip

(1) Assume that $k$ is a field. Let $Q$ be a finite quiver. Let $R$ the associated algebra of radical square zero and let $L(Q^{op})$ be the Leavitt path algebra of the opposite quiver $Q^{op}$. The following result was first obtained as \cite[Theorem 6.1]{ChenYang15}.

\begin{theorem*}[{\ref{thm:rad^2=0-vs-Leavitt}}] 
The singularity category $\cd_{sg}(R)$ is triangle equivalent to the perfect derived category $\per(L(Q^{op}))$ of $L(Q^{op})$. Here we consider the graded algebra $L(Q^{op})$ as a differential graded algebra with trivial differential.
\end{theorem*}

(2) Assume that $k$ is an algebraically closed field of characteristic $0$.
Let $\tilde{Q}$ be an Euclidean quiver with vertex set $\{0,1,\ldots,n\}$ such that $0$ is an extending vertex, and let $Q$ be the Dynkin quiver obtained from $\tilde{Q}$ by deleting $0$. For $\tilde{\lambda}\in k^{n+1}$, let $R$ be the corresponding deformed Kleinian singularity in the sense of \cite{CrawleyBoeveyHolland98}. When $\tilde{\lambda}=0$, $R$ is (the coordinate ring of) the Kleinian singularity associated to $Q$ . Let $\lambda$ be obtained from $\tilde{\lambda}$ by removing the $0$-th entry. Let $Q_{\lambda}$ be the full subquiver of $Q$ consisting of vertices $i$ with $\lambda_i=0$ and write $Q_{\lambda}=Q^{(1)}\cup\ldots \cup Q^{(s)}$ as the disjoint union of Dynkin quivers.   The following theorem was first obtained as \cite[Theorem 4.4]{Crawford16}. 

\begin{theorem*}[{\ref{thm:sing-cat-of-deformed-preproj-alg}}] Assume that $\lambda$ is dominant in the sense of \cite{CrawleyBoeveyHolland98}. Then the singularity category $\cd_{sg}(R)$ is triangle equivalent to $\cd_{sg}(R^{(1)})\oplus\ldots\oplus\cd_{sg}(R^{(s)})$, where $R^{(i)}$ is the Kleinian singularity corresponding to the Dynkin quiver $Q^{(i)}$.
\end{theorem*}

(3) Assume $k=\mathbb{C}$. Let $G\subset SL_3(\mathbb{C})$ be a finite subgroup, which naturally acts on $\mathbb{C}[x,y,z]$. Let $R=\mathbb{C}[x,y,z]^G$ be the ring of $G$-invariant polynomials. If $R$ has isolated singularity, then the `complete' version of the following theorem was obtained in \cite{deVolcseyVandenBergh16}; under suitable assumptions, including that $R$ has isolated singularity, the following theorem is very close to \cite[Corollary 5.3]{AmiotIyamaReiten15}.

\begin{theorem*}[{\ref{thm:singularity-category-of-quotient-singularity}}]  
The singularity category $\cd_{sg}(R)$ is triangle equivalent to the small cluster category of some quiver with potential.
\end{theorem*}

(4) Assume $k=\mathbb{C}$. Let $R$ be the coordinate ring of a Gorenstein affine toric threefold. Under suitable assumptions, including that $R$ has isolated singularity, the following result is very close to \cite[Theorem 6.3]{AmiotIyamaReiten15}.

\begin{theorem*}[{\ref{thm:singularity-category-of-toric-threefold}}]  The singularity category $\cd_{sg}(R)$ is triangle equivalent to the small cluster category of some quiver with potential.
\end{theorem*}

\subsubsection*{Notation}
Throughout let $k$ be a commutative ring.  We use $\ten$ to denote the tensor product over $k$. We compose morphisms in the same way as we compose functions, that is, $gf$ means $f$ followed by $g$. All functors between $k$-categories will be required $k$-linear.

For a $k$-algebra $A$ denote by $\Mod A$ the category of (right) $A$-modules, by $\mod A$ the category of finitely generated $A$-modules, and by $\proj A$ the category of finitely generated projective $A$-modules. Fix an injective generator $E$ of $\Mod k$ (we take $E=k$ if $k$ is a field) and let $D=\Hom_k(?,E)$.

Let $\ct$ be an additive $k$-category. It is said to be \emph{idempotent complete} if every idempotent morphism of $\ct$ has a kernel in $\ct$, and \emph{Hom-finite} if for any two objects $X$ and $Y$ the $k$-module $\Hom_\ct(X,Y)$ is of finite length. Let $\cs$ be a set of objects of $\ct$. Denote by $[\cs]$ the full subcategory of $\ct$ consisting of objects isomorphic to objects in $\cs$, and by $\add(\cs)=\add_\ct(\cs)$ the smallest full subcategory of $\ct$ containing $\cs$ and closed under taking direct summands and finite direct sums. We call $\cs$ a set of additive generators of $\add(\cs)$.

Let $\ct$ be a triangulated $k$-category and $\cs$ a set of objects of $\ct$. Denote by $\thick(\cs)=\thick_\ct(\cs)$ the smallest triangulated subcategory of $\ct$ containing $\cs$ and closed under taking direct summands and finite direct sums and by $\Loc(\cs)=\Loc_\ct(\cs)$ the smallest triangulated subcategory of $\ct$ containing $\cs$ and closed under taking direct summands and all existing direct sums in $\ct$. We call $\thick(\cs)$ the thick subcategory generated by $\cs$ and call $\cs$ a set of classical generators of $\thick(\cs)$. For two sets $\cs$ and $\cs'$ of objects of $\ct$, denote by $\cs*\cs'$ the full subcategory of $\ct$ consisting of objects $X$ such that there is a triangle $S\to X\to S'\to\Sigma S$ with $S\in\cs$ and $S'\in\cs'$.

\medskip
\noindent
\emph{Acknowledgement}. The authors thank Bernhard Keller and Michael Wemyss for answering their questions.

\section{DG quiver algebras}\label{s:cofibrant-model}

In this section we mainly recall some results on differential graded (=dg) algebras whose underlying graded algebra is the path algebra of a graded quiver, especially when an (ordinary) algebra is quasi-isomorphic to such a dg algebra.

\medskip

A \emph{dg $k$-algebra} $A$ is a $\mathbb{Z}$-graded $k$-algebra $A=\bigoplus_{p\in\mathbb{Z}}A^p$ endowed with a differential $d$ of degree $1$ such that the graded Leibniz rule holds
\[
d(ab)=d(a)b+(-1)^p ad(b)
\]
for all $a\in A^p$ ($p\in\mathbb{Z}$) and $b\in A$.

\smallskip
We assume that $k$ is a field in the rest of this section. 

\subsection{DG quiver algebras}\label{ss:dg-quiver-algebra}
Let $Q$ be a quiver. As usual, we denote by $Q_0$ the set of vertices of $Q$ and by $Q_1$ the set of arrows of $Q$. If both $Q_0$ and $Q_1$ are finite sets, we say that $Q$ is a finite quiver. For an arrow $\alpha$, we denote by $s(\alpha)$ and $t(\alpha)$ the source and the target of $\alpha$, respectively. A non-trivial path of $Q$ is a sequence $\alpha_1\cdots\alpha_l$ of arrows such that $s(\alpha_i)=t(\alpha_{i+1})$ for all $1\leq i\leq l-1$. For a non-trivial path $\alpha_1\cdots\alpha_l$ we define its source $s(\alpha_1\cdots\alpha_l)$ as $s(\alpha_l)$ and its target $t(\alpha_1\cdots\alpha_l)$ as $t(\alpha_1)$. For each vertex $i$ of $Q$, there is a trivial path $e_i$ with $s(e_i)=t(e_i)=i$. Assume that $Q_0$ is finite. The path algebra $kQ$ of $Q$ is the $k$-algebra which has  basis all paths of $Q$ and whose multiplication is given by concatenation of paths, \ie for two paths $p$ and $q$ of $Q$,
\[
p\cdot q=\begin{cases} pq & \text{if } s(p)=t(q),\\ 0 & \text{otherwise.}\end{cases}
\]

Let $Q$ be a graded quiver, \ie a quiver such that each arrow $\alpha$ of $Q$ is assigned with an integer $|\alpha|$. Assume that $Q_0$ is finite. Then the path algebra $kQ$ of $Q$ is naturally graded with the degree of a non-trivial path $\alpha_1\cdots\alpha_l$ being $|\alpha_1|+\ldots+|\alpha_l|$ and the degree of any trivial path being $0$. We will often call it the \emph{graded path algebra} of $Q$. 

\smallskip

Let $r\in\mathbb{N}$ and $K$ be the direct product of $r$ copies of $k$ (with standard basis $e_1,\ldots,e_r$). Let $V$ be a graded $K$-$K$-bimodule. Then the tensor algebra $T_K V:=\bigoplus_{p\geq 0} V^{\ten_K p}$ is isomorphic to the graded path algebra $kQ$ of the graded quiver $Q$ which has vertex set $\{1,\ldots,r\}$ and which has $\dim_k e_jV^m e_i$ arrows of degree $m$ from $i$ to $j$. 

\medskip
Following \cite{Oppermann17}, we call a dg $k$-algebra a \emph{dg quiver algebra}  if it is of the form $A=(kQ,d)$, where $Q$ is a
graded quiver with finitely many vertices and
the differential $d$ takes all
trivial paths to $0$.  By the graded Leibniz rule, the differential $d$ is determined by its value on arrows and it takes an arrow $\alpha$ to a linear combination of paths with source $s(\alpha)$ and target $t(\alpha)$.

\subsection{Deleting a vertex} \label{ss:deleting-a-vertex}

Let $A=(kQ,d)$ be a dg quiver algebra  and $i$ a vertex of $Q$. Then $A'=A/AeA$ is again a dg quiver algebra. Its quiver is obtained from $Q$ by deleting the vertex $i$ and its differential is obtained from $d$ by removing all the summands involving paths passing through $i$.

\begin{lemma}\label{lem:deleting-a-contractible-vertex} Keep the notation as above.
Assume that there is $x\in e_ikQe_i$ such that $d(x)=e_i$. Then the natural surjective dg algebra homomorphism $A\to A'$ is a quasi-isomorphism.
\end{lemma}
\begin{proof} 
It suffices to show that the kernel $V:=kQ e_i kQ$ is contractible. Observe that $V$ has a basis consisting of all paths of $Q$ passing through $i$. Define $h\colon V\to V$ as the unique $k$-linear map taking a path $\alpha_1\cdots\alpha_l$ to $(-1)^{|\alpha_1|+\ldots+|\alpha_r|}\alpha_1\cdots\alpha_rx\alpha_{r+1}\cdots\alpha_l$, where $r$ is the minimal integer such that the source of $\alpha_r$ is $i$. Then 
\begin{align*}
(d_V\circ h&+h\circ d_V)(\alpha_1\cdots\alpha_l)=d_V(h(\alpha_1\cdots\alpha_l))+h(d_V(\alpha_1\cdots\alpha_l))\\
&=d_V((-1)^{|\alpha_1|+\ldots+|\alpha_r|}\alpha_1\cdots\alpha_rx\alpha_{r+1}\cdots\alpha_l)\\
&\qquad+h(\sum_{j=1}^l(-1)^{|\alpha_1|+\ldots+|\alpha_{j-1}|}\alpha_1\cdots\alpha_{j-1}d(\alpha_j)\alpha_{j+1}\cdots\alpha_l)\\
&=(-1)^{|\alpha_1|+\ldots+|\alpha_r|}(\sum_{j=1}^r(-1)^{|\alpha_1|+\ldots+|\alpha_{j-1}|}\alpha_1\cdots\alpha_{j-1}d(\alpha_j)\alpha_{j+1}\cdots\alpha_rx\alpha_{r+1}\cdots\alpha_l\\
&\qquad+(-1)^{|\alpha_1|+\ldots+|\alpha_r|}\alpha_1\cdots\alpha_rd(x)\alpha_{r+1}\cdots\alpha_l\\
\end{align*}
\begin{align*}
&\qquad+\sum_{j=r+1}^l(-1)^{|\alpha_1|+\ldots+|\alpha_r|+1+|\alpha_{r+1}+\ldots+\alpha_{j-1}}\alpha_1\cdots\alpha_rx\alpha_{r+1}\cdots\alpha_{j-1}d(\alpha_j)\alpha_{j+1}\cdots\alpha_l)\\
&+(\sum_{j=1}^r(-1)^{|\alpha_1|+\ldots+|\alpha_{j-1}|}(-1)^{|\alpha_1|+\ldots+|\alpha_r|+1}\alpha_1\cdots\alpha_{j-1}d(\alpha_j)\alpha_{j+1}\cdots\alpha_rx\alpha_{r+1}\cdots\alpha_l\\
&\qquad+\sum_{j=r+1}^l(-1)^{|\alpha_1|+\ldots+|\alpha_{j-1}|}(-1)^{|\alpha_1|+\ldots+|\alpha_r|}\alpha_1\cdots\alpha_rx\alpha_{r+1}\cdots\alpha_{j-1}d(\alpha_j)\alpha_{j+1}\cdots\alpha_l)\\
&=\alpha_1\cdots\alpha_l.
\end{align*}
Therefore $\id_V$ is null-homotopic, and thus $V$ is contractible.
\end{proof}

\subsection{Resolutions of algebras}\label{ss:cofibrant-model}

Let $A=kQ/I$ be a $k$-algebra, where $Q$ is a finite quiver and $I$ is an ideal contained in the square of the ideal generated by arrows. For example, any finite-dimensional $k$-algebra is of this form. Let $R$ be a system of relations for $I$ in the sense of \cite[Section 1.2]{Bongartz83} and we assume that $R$ is finite. 
The following result is from the proof of~\cite[Theorem 6.10]{Keller11} and~\cite{Keller11c} (see also \cite[Construction 2.6 and Remark 2.9]{Oppermann17}). 

\begin{lemma}\label{lem:non-complete-cofibrant-model}
Keep the notation in the preceding paragraph. Then there is a quasi-isomorphism $\rho\colon \tilde{A}\to A$ of dg algebras, where $\tilde{A}=(k\tilde{Q},d)$ is a dg quiver algebra, such that $\tilde{Q}$ is a graded quiver which has the same vertices as $Q$ and whose arrows are concentrated in non-positive degrees such that the arrows of degree $0$ are precisely the arrows of $Q$, the arrows of degree $-1$ are of the form $\rho_r\colon s(r)\to t(r)$ for $r\in R$ and the differential $d$ takes $\rho_r$ to $r$ for $r\in R$.

Assume in addition that $A$ has global dimension $2$ and that $R$ consists of a set of representatives of $I/(I\m+\m I)$, where $\m$ is the ideal of $kQ$ generated by arrows.  Then $\tilde{Q}$ is concentrated in degrees $0$ and $-1$.
\end{lemma}

\begin{remark}\label{rem:cofibrant-model-for-Koszul-algebra}
Assume that all elements of $R$ are homogeneous with respect to path lengths. Then we can make $A$ a graded algebra by putting all arrows of $Q$ in degree $1$. If as such a graded algebra $A$ is Koszul (see \cite[Definition 1.2.1]{BeilinsonGinzburgSoergel96}), then the following equality holds by \cite[Theorem 2.10.1]{BeilinsonGinzburgSoergel96}
\begin{align}
\label{eq:minimal-relations}
\dim_k \Ext^2_A(S_j,S_i)=\# e_jRe_i,
\end{align}
for any pair of vertices $(i,j)$ of $Q$, where $S_i=A/(1-e_i)$. It follows that $R$ consists of a set of representatives of $I/(I\m +\m I)\cong \bigoplus_{i,j\in Q_0}D\Ext^2_A(S_j,S_i)$.
\end{remark}

\begin{example}\label{ex:s(2,2)}
Let $A=kQ/(R)$, where $Q$ is the quiver
\[
\xymatrix@C=3pc{
1\ar@<.5ex>[r]^{\alpha}&2\ar@<.5ex>[l]^{\beta}
}
\]
and $R=\{\alpha\beta\}$. The global dimension of $A$ is $2$. Applying Lemma~\ref{lem:non-complete-cofibrant-model}, we obtain the dg quiver algebra $(k\tilde{Q},d)$, where $\tilde{Q}$ is the quiver
\[
\xymatrix@C=3pc{
1\ar@<.5ex>[r]^{\alpha}&2\ar@<.5ex>[l]^{\beta}\ar@{-->}@(ur,dr)^\gamma
}
\]
with $\deg(\alpha)=\deg(\beta)=0$ and $\deg(\gamma)=-1$. The differential $d$ takes $\gamma$ to $\alpha\beta$ and takes $\alpha$ and $\beta$ to $0$.
\end{example}

\section{Derived categories of dg algebras}

In this section, we follow \cite{Keller94,Keller06d} to recall derived categories of dg algebras.

\subsection{Derived categories}
Let $A$ be a dg $k$-algebra. We view $A$ as a dg category with one object whose endomorphism algebra is $A$ and apply results on dg categories to $A$ without further remark.

A \emph{dg $A$-module} is a (right) graded module over the graded algebra $A$ endowed with a differential $d_M$ of degree $1$ such that the graded Leibniz rule holds
\[
d_M(ma)=d_M(m)a+(-1)^p md(a)
\]
for all $m\in M^p$ ($p\in\mathbb{Z}$) and $a\in A$.
For two (right) dg $A$-modules $M$ and $N$, define the complex $\cHom_A(M,N)$ componentwise as
\begin{eqnarray*}\cHom_A^p(M,N)=\left.\left\{f\in\prod_{q\in\mathbb{Z}}\Hom_k(M^q,N^{p+q}) \, \right|
\, f(ma)=f(m)a\right\},\end{eqnarray*} with
differential given by 
\[
d(f)=d_N\circ f-(-1)^p f\circ d_M\]
for $f\in\cHom_A^p(M,N)$. For example, $\cHom_A(A,N)=N$. The complex $\cEnd_A(M):=\cHom_A(M,M)$ with the
composition of maps as multiplication is a dg $k$-algebra. The \emph{dg category of dg $A$-modules} $\cc_{dg}(A)$ has dg $A$-modules as objects and has $\cHom_A(M,N)$ as morphisms between dg $A$-modules $M$ and $N$. The \emph{homotopy category of dg $A$-modules} $\ch(A)$ is the homotopy category of $\cc_{dg}(A)$, \ie the objects of $\ch(A)$ are the same as $\cc_{dg}(A)$, and the morphism space in $\ch(A)$ between two dg $A$-modules $M$ and $N$ is defined as $\Hom_{\ch(A)}(M,N)=H^0\cHom_A(M,N)$. The \emph{derived category of dg $A$-modules} $\cd(A)$ is the triangle quotient of $\ch(A)$ by the subcategory of acyclic dg $A$-modules, \ie those dg $A$-modules with trivial cohomologies. Denote by $\pi=\pi_A\colon \ch(A)\to \cd(A)$ the projection functor.

\subsection{$\ch$-projective and $\ch$-injective dg modules} Let $A$ be a dg $k$-algebra. 
A dg $A$-module $M$ is said to be \emph{$\ch$-projective} (respectively, \emph{$\ch$-injective}) if $\cHom_A(M,?)$ (respectively, $\cHom_A(?,M)$) preserves acyclicity. It is easy to see that $A_A$ is $\ch$-projective and $D({}_AA)$ is $\ch$-injective. Let $M$ be a dg $A$-module. By \cite[Theorem 3.1]{Keller94}, there is a quasi-isomorphism $\mathbf{p}M\to M$ of dg $A$-modules with $\mathbf{p}M$ being $\ch$-projective. This extends to a triangle functor $\mathbf{p}=\mathbf{p}_A\colon \cd(A)\to \ch(A)$, which is also denoted by $\mathbf{p}$, which is left adjoint to $\pi$. By \cite[Theorem 3.2]{Keller94}, there is a quasi-isomorphism $M\to \mathbf{i}M$ of dg $A$-modules with $\mathbf{i}M$ being $\ch$-injective. This extends to a triangle functor $\mathbf{i}=\mathbf{i}_A\colon \cd(A)\to \ch(A)$, which right adjoint to $\pi$. In particular, there are canonical isomorphisms
\begin{align*}
\Hom_{\cd(A)}(M,N)&\cong\Hom_{\ch(A)}(M,\mathbf{i}N)=H^0\cHom_A(M,\mathbf{i}N)\\
&\cong\Hom_{\ch(A)}(\mathbf{p}M,N)=H^0\cHom_A(\mathbf{p}M,N).
\end{align*}

\subsection{Subcategories} Let $A$ be a dg $k$-algebra. 
Let $\per(A)=\thick_{\cd(A)}(A_A)$ be the thick subcategory of $\cd(A)$ generated by $A_A$. When $A$ is a $k$-algebra (viewed as a dg algebra concentrated in degree $0$), a dg $A$-module is exactly a complex of $A$-modules, and $\cd(A)=\cd(\Mod A)$, the derived category of the abelian category $\Mod A$. In this case, $\per(A)$ is triangle equivalent to $
\ch^b(\proj A)$, the homotopy category of bounded complexes of finitely generated projective $A$-modules. When $k$ is a field, let $\cd_{fd}(A)$ be the full subcategory of $\cd(A)$ consisting of dg $A$-modules $M$ whose total cohomology $H^*(M)$ is finite-dimensional over $k$. If $A$ is a finite-dimensional algebra over the field $k$, then the canonical functor $\cd^b(\mod A)\rightarrow \cd_{fd}(A)$ is a triangle equivalence.

\begin{lemma}\label{lem:derived-equivalence-restricts-to-per-and-dfd}
Let $B$ be another dg $k$-algebra and let $F\colon \cd(B)\to\cd(A)$ be a triangle equivalence. Then $F$ restricts to a triangle equivalence $\per(B)\to\per(A)$. If $k$ is a field, then $F$ restricts to a triangle equivalence $\cd_{fd}(B)\to\cd_{fd}(A)$.
\end{lemma}
\begin{proof}
It is known that a triangle equivalence restricts to a triangle equivalence of the full subcategory of compact objects. The first statement then follows from the fact that $\per(A)$ is the full subcategory of $\cd(A)$ of compact objects, see \cite[Remark 5.3(a)]{Keller94}.

Assume that $k$ is a field. It follows by d\'evissage that a dg $A$-module $M$ belongs to $\cd_{fd}(A)$ if and only if $\bigoplus_{p\in\mathbb{Z}}\Hom_{\cd(A)}(P,\Sigma^p M)$ is finite-dimensional over $k$ for any $P\in\per(A)$. The second statement follows from this fact together with the first statement.
\end{proof}

For a dg $A$-module $M$, each cohomology $H^p(M)$ admits an $H^0(A)$-module structure. Let $\cd_{fg}(A)$ denote the full subcategory of $\cd(A)$ consisting of dg $A$-modules $M$ such that $H^*(M)$ is finitely generated over $H^0(A)$. If $H^0(A)$ is right noetherian as a ring, then this is a triangulated subcategory of $\cd(A)$. If $A$ is a $k$-algebra and is right noetherian as a ring, then the canonical functor $\cd^b(\mod A)\rightarrow\cd_{fg}(A)$ is a triangle equivalence. 
When $k$ is a field and $H^0(A)$ is finite-dimensional over $k$, $\cd_{fg}(A)$ is the same as $\cd_{fd}(A)$.

\subsection{Standard functors} Let $A$ and $B$ be dg $k$-algebras.
Let $M$ be a dg $B$-$A$-bimodule (\ie a dg $B^{op}\ten A$-module). Then we have an adjoint pair of dg functors 
\[
\begin{xy}
\SelectTips{cm}{}
\xymatrix{\cc_{dg}(B)\ar@<.7ex>[rr]^{?\ten_B M}&&\cc_{dg}(A)\ar@<.7ex>[ll]^{\cHom_A(M,?)}.}
\end{xy}
\]
If $M$ is $\ch$-projective over $A$, then $?\ten_B M$ preserves $\ch$-projectivity.
This adjoint pair of dg functors induces an adjoint pair of triangle functors
\[
\begin{xy}
\SelectTips{cm}{}
\xymatrix{\ch(B)\ar@<.7ex>[rr]^{?\ten_B M}&&\ch(A)\ar@<.7ex>[ll]^{\cHom_A(M,?)}.}
\end{xy}
\]
If $M$ is $\ch$-projective over $A$, then $\cHom_A(M,?)$ preserves acyclicity, so it induces a triangle functor $\cd(A)\to\cd(B)$, still denoted by $\cHom_A(M,?)$. If $M$ is $\ch$-projective over $B^{op}$, then $?\ten_B M$ preserves acyclicity, so it induces a triangle functor $\cd(B)\to\cd(A)$, still denoted by $?\ten_B M$. In general, we obtain adjoint derived functors
\[
\begin{xy}
\SelectTips{cm}{}
\xymatrix{\cd(B)\ar@<.7ex>[rr]^{?\lten_B M}&&\cd(A)\ar@<.7ex>[ll]^{\RHom_A(M,?)}}
\end{xy}
\]
as the following compositions
\[
\begin{xy}
\SelectTips{cm}{}
\xymatrix{\cd(B)\ar@<.7ex>[rr]^{\mathbf{p}_B}&&\ch(B)\ar@<.7ex>[rr]^{?\ten_B M}\ar@<.7ex>[ll]^{\pi_B}&&\ch(A)\ar@<.7ex>[ll]^{\cHom_A(M,?)}\ar@<.7ex>[rr]^{\pi_A}&&\cd(A).\ar@<.7ex>[ll]^{\mathbf{i}_A}
}
\end{xy}
\]
If $M$ is $\ch$-projective over $A$, then $\RHom_A(M,?)$ is isomorphic to $\cHom_A(M,?)$. If $M$ is $\ch$-projective over $B^{op}$, then $?\lten_B M$ is isomorphic to $?\ten_B M$.

Let $M$ be an $\ch$-projective or $\ch$-injective dg $A$-module and put $B=\cEnd_A(M)$. Then $M$ becomes a dg $B$-$A$-bimodule. 
By \cite[Lemma 4.2]{Keller94}, the adjoint pair 
\[
\begin{xy}
\SelectTips{cm}{}
\xymatrix{\cd(B)\ar@<.7ex>[rr]^{?\lten_B M}&&\cd(A)\ar@<.7ex>[ll]^{\RHom_A(M,?)}}
\end{xy}
\]
restricts to triangle equivalences
\[
\begin{xy}
\SelectTips{cm}{}
\xymatrix{\per(B)\ar@<.7ex>[rr]^(0.4){?\lten_B M}&&\thick_{\cd(A)}(M)\ar@<.7ex>[ll]^(0.55){\RHom_A(M,?)},}
\end{xy}
\]
and in the case when $M\in\per(A)$, to triangle equivalences
\[
\begin{xy}
\SelectTips{cm}{}
\xymatrix{\cd(B)\ar@<.7ex>[rr]^(0.4){?\lten_B M}&&\Loc_{\cd(A)}(M)\ar@<.7ex>[ll]^(0.55){\RHom_A(M,?)}.}
\end{xy}
\]

\subsection{Quasi-equivalences}\label{ss:quasi-equiv}
Let $A$ and $B$ be two dg $k$-algebras. A \emph{quasi-functor} $X\colon  B\to A$ is a dg $B$-$A$-bimodule $X$ such that $?\lten_B X\colon \cd(B)\to \cd(A)$ restricts to a functor $[B_B]\to[A_A]$, equivalently, there is an element $x\in Z^0(X)$ such that the map $A\to X, ~a\to xa$, is a quasi-isomorphism of dg $A$-modules. 
It is a \emph{quasi-equivalence} if in addition $?\lten_B X\colon \cd(B)\to \cd(A)$ is a triangle equivalence, equivalently, there is an element $x\in Z^0(X)$ such that the map $A\to X,~a\to xa$, is a quasi-isomorphism of dg $A$-modules and the map $B\to X,~a\to bx$, is a quasi-isomorphism of dg $B^{op}$-modules. See \cite[Sections 7.1 and 7.2]{Keller94} for more equivalent conditions.

Typical examples of quasi-equivalence arise from different resolutions of the same dg module. Let $M$ be a dg $A$-module and let $M'$ and $M''$ be two $\ch$-projective or $\ch$-injective resolutions of $M$ over $A$. Put $B'=\cEnd_A(M')$ and $B''=\cEnd_A(M'')$. If $M''$ is $\ch$-projective or $M'$ is $\ch$-injective, then $X=\cHom_A(M'',M')$ is a quasi-equivalence from $B'$ to $B''$. We can take $x$ to be any fixed quasi-isomorphism $M''\to M'$.

\subsection{Non-positive dg algebras: the standard t-structure and finiteness}\label{ss:standard-t-co-t-str}
A dg $k$-algebra $A$ is said to be \emph{non-positive} if  the degree $p$ component $A^p$ vanishes for all $p>0$. See \cite{KellerNicolas11,BruestleYang13,KoenigYang14,SuYang16} for previous study on non-positive dg algebras.

\smallskip
Let $A$ be a non-positive dg $k$-algebra. The canonical projection $A\rightarrow H^0(A)$ is a homomorphism of
dg algebras. We view a module over $H^0(A)$ as a dg module over $A$ via this homomorphism.
This defines a natural functor $\Mod H^0(A)\rightarrow \cd(A)$. We will identify $\Mod H^0(A)$ with its essential image in $\cd(A)$.

Let $\cd^{\leq 0}$ (respectively, $\cd^{\geq 0}$) be the full subcategory of $\cd(A)$ consisting of those dg modules whose cohomologies are concentrated in non-positive (respectively, non-negative) degrees. Then $(\cd^{\leq 0},\cd^{\geq 0})$ is a $t$-structure on
$\cd(A)$. The functor $H^0\colon \cd(A)\to \Mod H^0(A)$ restricts to an equivalence from the heart to $\Mod H^0(A)$. The following result is a generalisation of  \cite[Propositions 2.1]{KalckYang16}. The proofs in \emph{loc.cit.} can be adapted.
The equality $\cp^{\leq 0}=\cp_{\leq 0}$ in (b) follows from \cite[Corollary 2.4]{KalckYang16}. Recall that the category $\cd_{fg}(A)$ is a triangulated subcategory of $\cd(A)$ provided that $H^0(A)$ is right noetherian as a ring.

\begin{proposition}\label{prop:standard-t-str}
Let $A$ be a non-positive dg $k$-algebra. Assume that $H^0(A)$ is right noetherian as a ring. 
\begin{itemize}
\item[(a)] 
 Let $\cd_{fg}^{\leq 0}=\cd^{\leq 0}\cap \cd_{fg}(A)$ (respectively, $\cd_{fg}^{\geq 0}=\cd^{
 \geq 0}\cap  \cd_{fd}(A)$). Then $(\cd_{fg}^{\leq 0},\cd_{fg}^{\geq 0})$ is a $t$-structure on
$\cd_{fg}(A)$ and the functor $H^0$ restricts to an equivalence from its heart to $\mod H^0(A)$. The truncation functors are the standard truncations $\sigma^{\leq 0}$ and $\sigma^{\geq 0}$. 
Moreover, this is a bounded $t$-structure, \ie $\cd_{fg}(A)=\thick(\mod H^0(A))$.
\item[(b)] Assume further that $\cd_{fg}(A)\subseteq\per(A)$. Let $\cp^{\leq 0}=\cd^{\leq 0}\cap \per(A)$ (respectively, $\cp^{\geq 0}=\cd^{\geq 0}\cap \per(A)$). Then $\cp^{\leq 0}=\cp_{\leq 0}$, $\cp^{\geq 0}=\cd_{fg}^{\geq 0}$, and $(\cp^{\leq 0},\cp^{\geq 0})$ is a $t$-structure on
$\per(A)$ which has the same heart as $(\cd_{fg}^{\leq 0},\cd_{fg}^{\geq 0})$.
\end{itemize}
\end{proposition}

The following result is a generalisation of  \cite[Propositions 2.5]{KalckYang16}. The proof in \emph{loc.cit.} can be adapted.

 \begin{proposition}\label{prop:hom-finiteness}
 Let $A$ be a non-positive dg $k$-algebra. Assume that $H^0(A)$ is right noetherian as a ring and that $\cd_{fg}(A)\subseteq\per(A)$.
Then $H^p(A)$ is finitely generated over $H^0(A)$ for any $p\in\mathbb{Z}$.
\end{proposition}

\section{Universal localisation of graded hereditary algebras}\label{s:universal-localisation-of-graded-algebras}

In this section we recall the definition of recollement of triangualted categories and show that universal localisation of graded hereditary algebras induces localisation of derived categories of dg modules.

\subsection{Recollement}\label{ss:recollement}

A \emph{recollement} of triangulated $k$-categories in the sense of \cite{BeilinsonBernsteinDeligne82} is a diagram of $k$-linear triangle functors
\begin{align*}
\begin{xy}
\SelectTips{cm}{}
\xymatrix{\ct'\ar[rr]|{\, \, i_*=i_! \, \,}&&\ct\ar[rr]|{\, \,  j^!=j^* \, \, }\ar@/^15pt/[ll]|{\, \, i^! \, \,}\ar@/_15pt/[ll]|{\, \, i^*\, \, }&&\ct''\ar@/^15pt/[ll]|{\, \, j_* \, \,}\ar@/_15pt/[ll]|{\, \, j_! \, \,}}
\end{xy}
\end{align*}
such that 
\begin{enumerate}
\item $(i^*,i_*)$,\,$(i_!,i^!)$,\,$(j_!,j^!)$ ,\,$(j^*,j_*)$
are adjoint pairs;

\item  $i_*,\,j_*,\,j_!$  are full embeddings;

\item  $i^!\circ j_*=0$ (and thus also $j^!\circ i_!=0$ and
$i^*\circ j_!=0$);

\item  for each $X\in \ct$ there are triangles
\[
\xymatrix@R=0.4pc{
i_! i^!(X)\ar[r]& X\ar[r]& j_* j^* (X)\ar[r]& \Sigma i_!i^!(X)\\
j_! j^! (X)\ar[r]& X\ar[r]& i_* i^*(X)\ar[r]& \Sigma j_!j^!(X).
}
\]
\end{enumerate}

We will need the following facts:
\begin{itemize}
\item[(a)] Let $f\colon A\to B$ be a homomorphism of dg $k$-algebras. Then there is an adjoint triple of triangle functors
\begin{align*}
\begin{xy}
\SelectTips{cm}{}
\xymatrix{\cd(B)\ar[rr]|{\, \, f^* \, \,}&&\cd(A)\ar@/^15pt/[ll]|{\, \, f_\star \, \,}\ar@/_15pt/[ll]|{\, \, f_*\, \, }}
\end{xy},
\end{align*}
where $f^*=?\lten_B B=\RHom_B(B,?)$, $f_*=?\lten_A  B$ and $f_\star=\RHom_A(B,?)$, with $B$ considered a left dg $A$-module by $a\cdot b=f(a)b$ and a right dg $A$-module by $b\cdot a=bf(a)$. If $f$ is a \emph{homological epimorphism} of dg algebras \cite[Section 4]{NicolasSaorin09}, that is, if $f^*$ is fully faithful, then this diagram is the left half of a recollement, see \cite[Section 4]{NicolasSaorin09}.
\item[(b)] Let $A$ be a dg $k$-algebra and $M\in\per(A)$. Then there is an adjoint triple of triangle functors
\begin{align*}
\begin{xy}
\SelectTips{cm}{}
\xymatrix{
\cd(A)\ar[rr]&&\Loc(M)\ar@/^15pt/[ll]\ar@/_15pt/[ll]\\
\mbox{}
}
\end{xy},
\end{align*}
where the top functor is the inclusion. This diagram is the right half of a recollement, see for example \cite[Propositions 4.4.14 and 4.4.16]{Nicolas07t}.
\end{itemize}

\subsection{Universal localisation of graded hereditary algebras}
Let $A$ be a graded $k$-algebra. Let $\Grmod A$ denote the category of graded $A$-modules. A complex of graded $A$-modules can be considered as a bicomplex and taking total complexes defines a triangle functor $\Tot\colon \cd(\Grmod A)\to \cd(A)$, see \cite[Section 4]{KalckYang16a}.

Let $f\colon P^{-1}\to P^{0}$ be a morphism in $\grproj A$, the category of finitely generated graded projective $A$-modules

\begin{theorem}\label{thm:localisation-of-graded-alg}
There is a graded $k$-algebra $A_f$ together with a graded-algebra homomorphism $\varphi\colon A\to A_f$ such that $f\ten_A A_f$ is an isomorphism in $\grproj A_f$ and that $\varphi$ is universal with respect to this property. Moreover, if $A$ is right graded hereditary, then $A_f$ is also right graded hereditary.

Assume that $A$ is right graded hereditary. Then $\varphi$ induces a triangle equivalence up to direct summands
\[
\per(A)/\thick(\con(f))\longrightarrow \per(A_f).
\]
In other words, $\varphi$ induces a triangle equivalence
\begin{align*}
\left(\per(A)/\thick(\con(f))\right)^\omega{\longrightarrow}\per(A_f), 
\end{align*}
where $(-)^\omega$ denotes the idempotent completion (see \cite{BalmerSchlichting01}).
\end{theorem}
\begin{proof}
The first statement follows from \cite[Proposition 3.1]{ChenYang15}. 

Assume that $A$ is right graded hereditary. Then there is a commutative diagram
\begin{align*}
\begin{xy}
\SelectTips{cm}{}
\xymatrix@R=3pc@C=2pc{\cd(\Grmod A_f)\ar[d]^{\Tot}\ar[rr]|{\, \, \varphi^* \, \,}&&\cd(\Grmod A)\ar[d]^{\Tot}\ar[rr]\ar@/^15pt/[ll]|{\, \, \varphi_\star \, \,}\ar@/_15pt/[ll]|{\, \, \varphi_*\, \, }&&\Loc(\con(f)\langle p\rangle|p\in\mathbb{Z})\ar[d]^{\Tot}\ar@/^15pt/[ll]\ar@/_15pt/[ll]\\
\cd(A_f)\ar[rr]|{\, \, \varphi^*\, \,}&&\cd(A)\ar[rr]|{\, \, j^*\,\,}\ar@/^15pt/[ll]|{\, \, \varphi_\star \, \,}\ar@/_15pt/[ll]|{\, \, \varphi_* \, \, }&&\Loc(\con(f)),\ar@/^15pt/[ll]\ar@/_15pt/[ll]}.
\end{xy}
\end{align*}
where the right upper functor in each row is the inclusion.
By a graded-algebra version of \cite[Corollary 3.8]{ChenXi12}, the top row is a recollement. We claim that the bottom row is also a recollement. Then the desired equivalence follows from \cite[Theorem 2.1]{Neeman92a}. Now we verify the four conditions in the definition of a recollement.  By \cite[Theorem 5.1(e)]{KalckYang16a}, both $\Tot\colon \cd(\Grmod A_f)\to\cd(A_f)$ and $\Tot\colon \cd(\Grmod A)\to \cd(A)$ are dense.

(1) and (2) follow from Section~\ref{ss:recollement} (a) and (b), except the condition that $\varphi^*\colon \cd(A_f)\to\cd(A)$ is fully faithful, which follows from \cite[Theorem 5.1(e)]{KalckYang16a} since $\varphi^*\colon \cd(\Grmod A_f)\to\cd(\Grmod A)$ is fully faithful.

(3) By the commutativity of the diagram and the fact that the top row is a recollement, we have $j^*\circ\varphi^*\circ\Tot=0$, which implies that $j^*\circ\varphi^*=0$, since $\Tot\colon \cd(\Grmod A_f)\to\cd(A_f)$ is dense.

(4) For $M\in\cd(A)$, there is $\tilde{M}\in\cd(\Grmod A)$ such that $\Tot(\tilde{M})\cong M$. We form the two triangles associated to $\tilde{M}$ in the top row and apply $\Tot$ to obtain the desired triangles associated to $M$ in the bottom row.
\end{proof}

\section{Singularity categories}
In this section we collect some results on the idempotent completeness and Hom-finiteness of singularity categories and show that singularity categories naturally appear  in the context of algebraic triangulated categories.

\smallskip

Let $R$ be a $k$-algebra which is right noetherian as a ring. Following \cite{Buchweitz87,Orlov04}, we define the \emph{singularity category} of $R$  as
\[
\cd_{sg}(R):=\cd^b(\mod R)/\ch^b(\proj R).
\]
By \cite[Theorem 3.8]{Beligiannis00} (or the dual of \cite[Exemple 2.3]{KellerVossieck87}), $\cd_{sg}(R)$ is triangle equivalent to the stabilisation of the additive quotient of $\mod R$ by $\proj R$. In general $\cd_{sg}(R)$ is neither idempotent complete nor Hom-finite.

\subsection{Relation to algebraic triangulated categories}

Let $\ce$ be a Frobenius $k$-category and denote by $\proj \ce\subseteq \ce$ the full subcategory of projective objects. We refer to Keller's overview article for definitions and unexplained terminology \cite{Keller96}.Denote by $\ul{\ce}$ the stable category of $\ce$, which has the same objects as $\ce$ and morphism spaces are defined as $\ul{\Hom}_{\ce}(X, Y)=\Hom_{\ce}(X, Y)/\cp(X, Y)$. Here $\cp(X, Y)$ consists of morphisms factoring through $\proj \ce$.  

\smallskip
Recall that a noetherian ring $R$ is called an \emph{Iwanaga--Gorenstein ring} if ${\rm inj.dim}_RR<\infty$ and ${\rm inj.dim}R_R<\infty$. Let
\[\MCM(R):=\{X\in\mod R\mid\Ext^i_R(X,R)=0\ \forall i>0\}\]
be the category of \emph{maximal Cohen--Macaulay $R$-modules}, which is a Frobenius $k$-category with $\proj \MCM(R)=\proj R$. 

\begin{theorem}[{\cite[Theorem 4.4.1(2)]{Buchweitz87}}]\label{thm:Buchweitz'-theorem}
Let $R$ be an Iwanaga--Gorenstein ring. Then there is a natural triangle equivalence
\[
\underline{\MCM}(R)\stackrel{\simeq}{\longrightarrow} \cd_{sg}(R).
\]
\end{theorem}

This theorem admits the following generalisation. Let $\ce$ be an idempotent complete Frobenius $k$-category.

\begin{proposition}\label{prop:Buchweitz}
Assume that $\proj \ce= \add P$ for some $P \in \proj \ce$.  If moreover $R=\End_{\ce}(P)$ is right noetherian as a ring, then there is a fully faithful triangle functor
$$\ul{\widetilde{\P}}\colon  \ul{\ce} \longrightarrow  \cd_{sg}(R).$$
\end{proposition}

We need some preparation. Taking projective resolutions yields a functor $\P\colon  \ce \ra \ch^{-,b}(\proj \ce) $, where for an exact category $\ca$ we denote by $\ch^{-, b}(\proj \ca)\subseteq \ch^-(\proj \ca)$ the full triangulated subcategory consisting of those complexes of projective objects $P^\bullet$, which are acyclic in small enough degrees (\ie there exists $N \in \Z$ such that there are conflations
$
\begin{xy}\SelectTips{cm}{}
\xymatrix{
Z^n(P^{\bullet}) \ar@{>->}[r]^(.65){i_{n}} & P^n \ar@{->>}[r]^(.35){p_{n}} & Z^{n+1}(P^\bullet) }
\end{xy}
$  in $\ca$ and such that $d^n_{P^\bullet}=i^{n+1}p^n$ for all $n<N$). The following lemma is dual to \cite[Exemple 2.3]{KellerVossieck87}.

\begin{lemma}\label{lem:StFrob}
The functor $\P$ induces an equivalence of triangulated categories $$\ul{\P}\colon  \ul{\ce} \longrightarrow \ch^{-, b}(\proj \ce)/\ch^b(\proj \ce).$$
\end{lemma}

\begin{proof}[Proof of Proposition~\ref{prop:Buchweitz}]
It is known that there is an equivalence of additive $k$-categories $\Hom_{\ce}(P, ?)\colon  \proj \ce \rightarrow \proj R.$ This induces a triangle equivalence $\ch^-(\proj \ce) \ra \ch^-(\proj R)$. The fully faithful restriction $\ch^{-, b}(\proj \ce) \ra \ch^{-, b}(\proj R)$ is well defined since $P$ is projective. Combining this with the exact equivalence $\ch^{-, b}(\proj R) \ra \cd^b(\mod R)$ and Lemma \ref{lem:StFrob} we obtain a fully faithful functor
$$
\begin{xy}\SelectTips{cm}{}
\xymatrix
{
\ul{\ce} \ar[r]^(.3){\displaystyle{\ul{\P}}} & \displaystyle{\frac{\ch^{-, b}(\proj \ce)}{\ch^b(\proj \ce)}} \ar[r] & \displaystyle{\frac{\ch^{-, b}(\proj R)}{\ch^b(\proj R)}} \ar[r]^{\simeq} & \displaystyle{\frac{\cd^b(\mod R)}{\ch^b(\proj R)}}.
}
\end{xy}
$$
\end{proof}

\begin{remark}
We point out that the restriction of the equivalence $\ch^-(\proj \ce) \ra \ch^-(\proj R)$ to $\ch^{-, b}(\proj \ce) \ra \ch^{-, b}(\proj R)$ is \emph{not} dense in general. The reason is that the two ambient exact categories $\proj \ce \subseteq \ce$ and $\proj R \subseteq \mod R$ may differ. This leads to different notions of acyclicity. As a consequence, the functor $\widetilde{\mathbb{P}}$ in Proposition~\ref{prop:Buchweitz} is not dense in general. However, it is dense if $\ce$ admits a non-commutative resolution, see \cite[Theorem 2.7]{KalckIyamaWemyssYang15}. In fact in this case $R$ has to be Iwanaga--Gorenstein and $\Hom_\ce(P,?)$ restricts to an equivalence $\ce\simeq\MCM(R)$.
\end{remark}

\subsection{Idempotent completeness}
Let $S$ be a commutative  $k$-algebra which is local complete noetherian as a ring.

\begin{lemma}\label{L:stable-complete}
Let $R$ be an Iwanaga--Gorenstein $S$-algebra which is finitely generated as an $S$-module. Then $\cd_{sg}(R)$ is idempotent complete.
\end{lemma}
\begin{proof} In view of Theorem~\ref{thm:Buchweitz'-theorem}, we will show that $\underline{\MCM}(R)$ is idempotent complete.

Let $M \in \MCM(R)$ without projective direct summands and $e \in \ul{\End}_{R}(M)$ be an idempotent endomorphism. We claim that there exists an idempotent $\epsilon \in \End_{R}(M)$, which is mapped to $e$ under the canonical projection $\End_{R}(M) \ra \ul{\End}_{R}(M)$. In other words we need some lifting property for idempotents. This is known to hold for any $S$-algebra $B$, which is finitely generated as an $S$-module and any two-sided ideal $I \subseteq \rad B$ \cite[Proposition 6.5 and Theorem 6.7]{CurtisReiner1}. To prove our claim it thus suffices to show $\cp(M) \subseteq \rad \End_{R}(M)$, where $\cp(M)$ is the two-sided ideal of endomoprhisms factoring through a projective $R$-module. Assume that there exists $f \in \cp(M) \setminus \rad \End_{R}(M)$. This means that there is a $g \in \End_{R}(M)$ such that $\id_{M} - gf$ is not invertible, hence not surjective \cite[Proposition 5.8]{CurtisReiner1}. Let $M=\bigoplus_{i=1}^t M_{i}$ be a decomposition of $M$ into indecomposable modules and denote by $\iota_{i}$ and $\pi_{j}$ the canonical inclusions and projections respectively. We can without loss of generality assume that there exists some $i$ such that $\pi_{i} (\id_{M} - gf)\colon M \ra M_{i}$ is not surjective. Hence, $\pi_{i} (\id_{M} - gf)\iota_{i}=\id_{M_{i}}-\pi_{i}gf\iota_{i}=\id_{M_{i}}-(g_{1i}f_{i1}+ \cdots + g_{ii}f_{ii}+ \cdots + g_{ti}f_{ti})$ is not surjective. Here, $f_{ij}=\pi_{j}f\iota_{i}$ and $g_{ij}=\pi_{j}g\iota_{i}$. Since $f \in \cp(M)$ we have $g_{ji}f_{ij} \in \cp(M_{i})$ for $j=1, \cdots, t.$ Since $M_{i}$ is indecomposable $\End_{R}(M_{i})$ is local \cite[Proposition 6.10]{CurtisReiner1}. It follows that $\cp(M_{i}) \subseteq \rad \End_{R}(M_{i})$, for otherwise $M_{i}$ is a direct summand of a projective module, which contradicts our assumptions on $M$. Hence, $s=\sum_{j=1}^t g_{ji}f_{ij} \in \rad \End_{R}(M_{i})$ and therefore $\id_{M_{i}}-s$ is an isomorphism. Contradiction. Thus $f \in \rad \End_{R}(M)$.

To show that $\ul{\MCM}(R)$ is idempotent complete, let $M \in \ul{\MCM}(R)$. We can assume that $M$ has no projective direct summands. Let $e \in \ul{\End}_{R}(M)$ be an idempotent endomorphism. By the considerations above $e$ lifts to an idempotent $\epsilon \in \End_{R}(M)$. $\MCM(R)$ is idempotent complete since it is closed under direct summands in the abelian and hence idempotent complete category $\mod R$. Thus we have a direct sum decomposition $M \cong N_{1} \oplus N_{2}$ such that $\epsilon=\iota_{1}\pi_{1}$, where $\iota_{1}$ and $\pi_{1}$ denote the canonical inclusion respectively projection of $N_{1}$. Projecting to the stable category yields the desired factorization of $e$.  
\end{proof}

\subsection{Hom-finiteness} Assume that $k$ is a field. 

\begin{theorem}\label{thm:Hom-finiteness-Murfet}
Let $R$ be a local commutative $k$-algebra which is noetherian as a ring. Assume that the  residue field is $k$ and that $R$ has isolated singularity. Then $\cd_{sg}(R)$ is Hom-finite if and only if $R$ is Gorenstein. 
\end{theorem}
\begin{proof}
Assume that $R$ is Gorenstein. Then $\cd_{sg}(R)$ is Hom-finite by \cite[Lemma 3.4]{Murfet13}. 

Assume that $\cd_{sg}(R)$ is Hom-finite. Then, the stable cohomology ${\widehat{\Ext}}^0_R(k,k)$ of $k$, being isomorphic to $\End_{\cd_{sg}(R)}(k)$ (\cite[Section 1.4.2]{AvramovVeliche07}),  is finite-dimensional over $k$, therefore $R$ is Gorenstein by \cite[Theorem 6.4]{AvramovVeliche07}.
\end{proof}

\begin{theorem}\label{thm:idempotent-completeness-and-Hom-finiteness-of-singularity-category}
Let $R$ be a Gorenstein commutative $k$-algebra. Assume that $R=\bigoplus_{p\in\mathbb{Z}}R^p$ admits the structure of a graded algebra such that $R^0=k$ and $R^p=0$ for $p<0$ and that $R$ has isolated singularity at $\m=\bigoplus_{p>0}R^p$. Then $\cd_{sg}(R)$ is idempotent complete and Hom-finite.
\end{theorem}
\begin{proof}
Let $\widehat{R}$ be the completion of $R$ with respect to the $\m$-adic topology on $R$. Then by \cite[Proposition A.8]{KellerMurfetVandenBergh11}  there is a triangle equivalence
\[
\cd_{sg}(R)\stackrel{\simeq}{\longrightarrow} \cd_{sg}(\widehat{R}).
\]
Since $\cd_{sg}(\widehat{R})$ is idempotent complete (Lemma~\ref{L:stable-complete})  and Hom-finite (Theorem~\ref{thm:Hom-finiteness-Murfet}, \cite[Theorem]{Auslander86}), 
the proof is complete.
\end{proof}

\section{Relative singularity categories}
\label{s:relative-singularity-category}

Let $A$ be a $k$-algebra and $e\in A$ an idempotent. We call the triangle quotient
\[
\ch^b(\proj A)/\thick(eA)
\] 
the \emph{singularity category of $A$ relative to $e$}. If $A$ is of finite global dimension, this is a relative singularity category in the sense of \cite{ChenXW11}.
In \cite[Section 2.8]{KalckYang16}, we showed that up to direct summands 
$\ch^b(\proj A)/\thick(eA)$ is triangle equivalent to $\per(B)$ for a nice dg algebra $B$. In this section, we discuss the relation between $\per(B)$ and the singularity category of $eAe$.

\subsection{DG model}

We first recall \cite[Proposition 2.10]{KalckYang16}.

\begin{proposition}\label{p:recollement-from-projective-general-case} 
Let $A$ be a flat $k$-algebra and $e\in A$ an idempotent. There is a
dg $k$-algebra $B$ with a homomorphism of dg $k$-algebras
$f\colon  A\rightarrow B$ and a recollement of derived categories
\begin{align}\label{E:RecollmentByIdempotent}
\begin{xy}
\SelectTips{cm}{}
\xymatrix@C=4pc{\cd(B)\ar[rr]|{\, \, f^* \, \,}&&\cd(A)\ar[rr]|{\, \,  ?\lten_A Ae \, \, }\ar@/^15pt/[ll]|{\, \, f_\star \, \,}\ar@/_15pt/[ll]|{\, \, f_*\, \, }&&\cd(eAe)\ar@/^15pt/[ll]|{\, \, \RHom_{eAe}(Ae,?) \, \,}\ar@/_15pt/[ll]|{\, \, ?\lten_{eAe} eA \, \,}}
\end{xy},
\end{align}
such that $B^i=0$ for $i>0$ and $H^0(B)\cong A/AeA$.
\end{proposition}

The homomorphism $f\colon A\to B$ in Proposition~\ref{p:recollement-from-projective-general-case} is a homological epimorphism (Section~\ref{ss:recollement}). In practice we will not restrict ourselves to such homomorphisms, but will also consider morphisms of the form
\[
\xymatrix{& A'\ar[ld]_s \ar[rd]^f\\
A && B
}
\]
where $s$ is a quasi-isomorphism of dg algebras and $f$ is a homological epimorphism of dg algebras. In this way, the flatness condition on $A$ in the preceding proposition can be removed, see \cite[page xxvi and the proof of Lemma 5.4.3]{Nicolas07t}.

\begin{lemma}\label{lem:B-in-homotopy-category-of-dg-algebras}
Keep the notation and assumptions in Proposition~\ref{p:recollement-from-projective-general-case}. Assume that there is a morphism $A\stackrel{s'}{\leftarrow}A'\stackrel{f'}{\rightarrow}B'$, where $s'\colon A'\to A$ is a quasi-isomorphism of dg algebras and $f'\colon A'\to B'$ is a homological epimorphism of dg algebras such that $\Im(f'^*)=\Loc(s'^*(eA))^\perp$. Then there is a triangle equivalence $\cd(B)\to \cd(B')$ taking $B_B$ to a dg $B'$-module quasi-isomorphic to $B'_{B'}$. If $k$ is a field, then $B$ and $B'$ are quasi-equivalent.
\end{lemma}
\begin{proof}
Under the assumptions, $s'^*\colon \cd(A)\to\cd(A')$ is a triangle equivalence and the equalities $\Im(f'^*)=\Loc(s'^*(eA))^\perp=\Im(fs')^*$ hold. So the composite functor
\[
\xymatrix{\cd(B)\ar[r]^{?\lten_B B} & \cd(A)\ar[r]^{?\lten_A B'} & \cd(B')}
\]
is a triangle equivalence which takes $B$ to a dg $B'$-module quasi-isomorphic to $B'_{B'}$. If $k$ is a field, then by \cite[Lemma 6.3(b)]{Keller94} we obtain that $B$ and $B'$ are quasi-equivalent. 
\end{proof}

\begin{remark}
The triangle equivalence $\cd(B)\to\cd(B')$  in Lemma~\ref{lem:B-in-homotopy-category-of-dg-algebras} restricts to a triangle equivalence $\per(B)\to\per(B')$ and, if $A/AeA$ is right noetherian as a ring, to a triangle equivalence $\cd_{fg}(B)\to\cd_{fd}(B')$. Therefore all results in the rest of this section remain true with $B$ replaced by $B'$.
\end{remark}

Next we recall and generalise results from \cite[Section 2.8]{KalckYang16} on properties of $B$.

\begin{corollary}\label{c:restriction-and-induction} Keep the notation and assumptions in Proposition~\ref{p:recollement-from-projective-general-case}.
\begin{itemize}
 \item[(a)] \emph{(\cite[Corollary 2.12(a)]{KalckYang16})}
The functor
$i^*$ induces a triangle equivalence up to direct summands
\begin{align} 
\ch^b(\proj A)/\thick(eA){\longrightarrow}\per(B).\nonumber 
\end{align}

 \item[(b)] Assume that $A/AeA$ is right noetherian as a ring. Let $\cd_{fg,A/AeA}(A)$ be the full subcategory of $\cd_{fg}(A)$ consisting of
complexes with cohomologies supported on $A/AeA$. The functor $i_*$
restricts to a triangle equivalence
$\cd_{fg}(B)\stackrel{\simeq}{\longrightarrow}\cd_{fg,A/AeA}(A)$ with quasi-inverse the restriction of $i^*$.
Moreover, the latter category coincides with
$\thick_{\cd(A)}(\mod A/AeA)$.
\item[(c)] Let $q\colon \ch^b(\proj A)\rightarrow \ch^b(\proj A)/\thick(eA)$ be the canonical quotient functor. Assume that $A/AeA$ is right noetherian as a ring and that all finitely generated $A/AeA$-modules have finite projective dimension over $A$. Then the triangle equivalence up to direct summands in (a) restricts to a triangle equivalence
\[
\thick(q(\mod A/AeA))\longrightarrow \cd_{fg}(B).
\]
\end{itemize}
\end{corollary}
\begin{proof}
(b) Similar to the proof of \cite[Corollary 2.12(b)]{KalckYang16}.

(c) Since $\thick(\mod A/AeA)$ is right perpendicular to $\thick(eA)$, it follows from for example \cite[Lemma 9.1.5]{Neeman99} that $q$ restricts to a triangle equivalence 
\[\thick(\mod A/AeA)\to \thick(q(\mod A/AeA)).\]
The desired result then follows immediately from (b).
\end{proof}

\begin{proposition}\label{p:finitely-generated-are-perfect}
We keep the notation and assumptions in Proposition~\ref{p:recollement-from-projective-general-case}. If $A/AeA$ is right noetherian as a ring and that all finitely generated $A/AeA$-modules have finite projective dimension over $A$, then
\begin{itemize}
\item[(a)] $\cd_{fg}(B)\subseteq\per(B)$;
\item[(b)] $H^i(B)$ is finitely generated over $H^0(B)$  for any $i\in\mathbb{Z}$;
\item[(c)] $\per(B)$ has a $t$-structure whose heart is equivalent to $\mod A/AeA$;
\item[(d)]
assume additionally that $k$ is a field and $A/AeA$ is finite-dimensional over $k$, then $\per(B)$ is Hom-finite.
\end{itemize}
\end{proposition}
\begin{proof} (a) We have $i^*i_*(\cd_{fg}(B))=\cd_{fg}(B)$ because $i_*$ is fully faithful, and $i^*(A)=B$.
Therefore
in order to show $\cd_{fg}(B)\subseteq\per(B)$ it suffices to show that $i_*(\cd_{fg}(B))\subseteq \ch^b(\proj A)$.
By Corollary~\ref{c:restriction-and-induction}(b),   $i_*(\cd_{fg}(B))=\thick_{\cd(A)}(\mod A/AeA)$.
By assumption $\mod-A/AeA\subseteq \ch^b(\proj A)$, and hence $\thick_{\cd(A)}(\mod A/AeA)\subseteq \ch^b(\proj A)$. Thus $\cd_{fg}(B)\subseteq \per(B)$.

(b) and (c) follow from Propositions~\ref{prop:standard-t-str} and~\ref{prop:hom-finiteness}.

(d) follows from \cite[Proposition 2.5]{KalckYang16}.
\end{proof}

\subsection{The relationship between the singularity category and the relative singularity category}

\begin{theorem}\label{t:singularity-category-vs-relative-singularity-category}
Let $A$ be a flat $k$-algebra which is right noetherian as a ring and let $e\in A$ be an idempotent such that all finitely generated $A/AeA$-modules have finite projective dimension over $A$. Let $B$ be obtained as in Proposition~\ref{p:recollement-from-projective-general-case}. 
Then there is a triangle equivalence up to direct summands
\begin{align}
 \thick_{\cd_{sg}(eAe)}(Ae) \longrightarrow \per(B)/\cd_{fg}(B),
\end{align}
which is a triangle equivalence if and only if $\ch^b(\proj A)/\thick(eA)$ is idempotent complete. If $\cd_{sg}(eAe)$ is idempotent complete, then this is a triangle equivalence, and the categories $\ch^b(\proj A)/\thick(eA)$ and $\per(B)/\cd_{fg}(B)$ are idempotent complete.

Assume that $A$ has finite global dimension. Then  $\thick_{\cd_{sg}(eAe)}(Ae)=\cd_{sg}(eAe)$ and there is a triangle equivalence up to direct summands
\begin{align}
\cd_{sg}(eAe)\longrightarrow \per(B)/\cd_{fg}(B).
\end{align}
If $\ch^b(\proj A)/\thick(eA)$ is idempotent complete, then this is a triangle equivalence.
\end{theorem}
\begin{proof} We adopt the notation in Proposition~\ref{p:recollement-from-projective-general-case} and Corollary~\ref{c:restriction-and-induction}.

By \cite[Proposition 3.3]{KalckYang16}, there is a triangle equivalence
\begin{align}
\frac{\cd^b(\mod A)/\thick(eA)}{\thick(q(\mod A/AeA))}\stackrel{F}{\longrightarrow} \frac{\cd^b(\mod eAe)}{\ch^b(\proj eAe)}=\cd_{sg}(eAe).\nonumber
\end{align}
It is induced from $\Hom_A(eA,?)\colon \cd^b(\mod A)\to\cd^b(\mod eAe)$. Since $\ch^b(\proj A)$ is the thick subcategory of $\cd^b(\mod A)$ generated by $A_A$ and $F(A_A)=Ae$, it follows that $F$ restricts to a triangle equivalence
\begin{align}
\frac{\ch^b(\proj A)/\thick(eA)}{\thick(q(\mod A/AeA))}\stackrel{\bar{F}}{\longrightarrow} \thick_{\cd_{sg}(eAe)}(Ae).\nonumber
\end{align}
By Corollary~\ref{c:restriction-and-induction},
there is a triangle equivalence up to direct summands
\begin{align}
\frac{\ch^b(\proj A)/\thick(eA)}{\thick(q(\mod A/AeA))}
\stackrel{G}{\longrightarrow} \per(B)/\cd_{fg}(B),\nonumber
\end{align}
which is an equivalence if and only if $\ch^b(\proj A)/\thick(eA)$ is idempotent complete. As a consequence, $G\circ \bar{F}^{-1}\colon \thick_{\cd_{sg}(eAe)}\to\per(B)/\cd_{fg}(B)$ is a triangle equivalence up to direct summands and it is an equivalence if and only if $\ch^b(\proj A)/\thick(eA)$ is idempotent complete. If $\cd_{sg}(eAe)$ is idempotent complete, then $\thick_{\cd_{sg}(eAe)}(Ae)$ is idempotent complete and $G\circ \bar{F}^{-1}$ has to be an equivalence. Therefore the categories $\ch^b(\proj A)/\thick(eA)$ and $\per(B)/\cd_{fg}(B)$ are idempotent complete.

\smallskip
Now assume that the global dimension of $A$ is finite. Then $\cd^b(\mod A)=\ch^b(\proj A)$, and hence the functors $F$ and $\bar{F}$ have to coincide. As a consequence, $\thick_{\cd_{sg}(eAe)}(Ae)$ coincides with $\cd_{sg}(eAe)$ and we obtain the second desired equivalence. 
\end{proof}

Put $\Delta_e(A):=\ch^b(\proj A)/\thick(eA)$. The following result shows that up to direct summands $\thick_{\cd_{sg}(eAe)}(Ae)$ is a canonical triangle quotient of $\Delta_e(A)$. For a triangulated $k$-category $\ct$, we denote by $\ct_{rhb}$ the full subcategory of $\ct$ consisting of the objects $X$ such that for any $Y\in\ct$, the space $\Hom(Y,\Sigma^p X)$ vanishes for almost all $p\in\mathbb{Z}$. 

\begin{corollary}\label{cor:relative-singularity-category-vs-singularity-category}
Let $A$ be a flat $k$-algebra which is right noetherian as a ring and let $e\in A$ be an idempotent such that all finitely generated $A/AeA$-modules have finite global dimension over $A$. Let $B$ be obtained as in Proposition~\ref{p:recollement-from-projective-general-case}. 
Then there is a triangle equivalence up to direct summands
\begin{align}
 \thick_{\cd_{sg}(eAe)}(Ae) \longrightarrow \Delta_e(A)/\Delta_e(A)_{rhb}.
\end{align}
If $A$ has finite global dimension, then there is a triangle equivalence up to direct summands
\begin{align}
\cd_{sg}(eAe)\longrightarrow \Delta_e(A)/\Delta_e(A)_{rhb}.
\end{align}
\end{corollary}
\begin{proof}
Since $\cd_{fg}(B)=\per(B)_{rhb}$, it follows that the triangle equivalence up to direct summands $\Delta_e(A)\to \per(B)$ in Corollary~\ref{c:restriction-and-induction}(a) restricts to a triangle equivalence up to direct summands $\Delta_e(A)_{rhb}\to \cd_{fg}(B)$. In view of Corollary~\ref{c:restriction-and-induction}(c), this is in fact a triangle equivalence. We are done by applying Theorem~\ref{t:singularity-category-vs-relative-singularity-category}.
\end{proof}

\subsection{Non-commutative resolutions}

Let $R$ be a commutative $k$-algebra which is noetherian as a ring.
A \emph{non-commutative resolution} (NCR for short) of $R$ is a $k$-algebra of the form $A=\End_R(R\oplus M)$, where $M$ is a finitely generated $R$-module, such that $A$ has finite global dimension. Let $e\in A$ be the idempotent of $A$ corresponding to the direct summand $R$ of $R\oplus M$.

\begin{corollary}\label{cor:dg-algebra-B-for-NCR}
Let $A$ be an NCR of $R$ and let $B=B_R$ be the dg algebra in Proposition~\ref{p:recollement-from-projective-general-case}. Then
\begin{itemize}
\item[(a)] $B^i=0$ for any $i>0$;
\item[(b)] $H^0(B)\cong A/AeA=\underline{\End}_R(M)$, the stable endomorphism algebra of $M$;
\item[(c)] $\cd_{sg}(R)=\thick_{\cd_{sg}(R)}(M)$;
\item[(d)] there is a triangle equivalence $\cd_{sg}(R)\simeq \per(B)/\cd_{fg}(B)$ up to direct summands, taking $M$ to $B$; if $\cd_{sg}(R)$ is idempotent complete (\eg if $R$ is local complete), then this is a triangle equivalence.
\end{itemize}
\end{corollary}
\begin{proof}
(a) and (b) follow from Proposition~\ref{p:recollement-from-projective-general-case}, and (c) and (d) follow from Theorem~\ref{t:singularity-category-vs-relative-singularity-category}.
\end{proof}

\subsection{Algebraic triangulated categories with a classical generator} \label{ss:alg-tria-cl-sing}
Let $\ce$ be an idempotent complete Frobenius $k$-category and assume that $\proj \ce= \add P$ for some $P \in \proj \ce$.  Let $R=\End_{\ce}(P)$. Let $M$ be a classical generator of $\underline{\ce}$ and assume that $M$, as an object of $\ce$, contains $P$ as a direct summand. Put $A=\End_\ce(M)$. Let $e$ be the identity endomorphism of $P$ and view it as an element of $A$. Let $B$ be the dg algebra obtained in Proposition~\ref{p:recollement-from-projective-general-case}.

\begin{corollary}
Keep the notation and assumptions in the preceding paragraph. Assume further that $A$ is right noetherian as a ring and that all finitely generated $A/AeA$-modules have finite projective dimension over $A$. Then
\begin{itemize}
\item[(a)] $B^i=0$ for any $i>0$;
\item[(b)] $H^0(B)\cong A/AeA={\End}_{\underline{\ce}}(M)$;
\item[(c)] there is a triangle equivalence $\underline{\ce}\simeq\thick_{\cd_{sg}(R)}(Ae)$ up to direct summands;
\item[(d)] there is a triangle equivalence $\underline{\ce}\simeq \per(B)/\cd_{fg}(B)$ up to direct summands, taking $M$ to $B$.
\end{itemize}
\end{corollary}

\begin{proof}
(a) and (b) follow from Proposition~\ref{p:recollement-from-projective-general-case}. (c) and (d) follow from Theorem~\ref{t:singularity-category-vs-relative-singularity-category}, because the fully faithful functor $\ul{\widetilde{\P}}\colon  \underline{\ce}\to\cd_{sg}(R)$ in Proposition~\ref{prop:Buchweitz} takes $M$ to $Ae$.
\end{proof}

\section{Construction of dg model}
\label{s:dg-model}
\label{ss:construction-of-B}
Let $A$ be a flat $k$-algebra and $e\in A$ an idempotent. In Section~\ref{s:relative-singularity-category}, we showed that up to direct summands $\ch^b(\proj A)/\thick(eA)$ is triangle equivalent to $\per(B)$ for a nice dg algebra $B$. We remark that $B$ is not unique but unique up to a nice equivalence (Lemma~\ref{lem:B-in-homotopy-category-of-dg-algebras}). In this section, we discuss various constructions of $B$. We first recall two constructions of $B$ from general theories, then we give a new construction, which is for special cases but will be very useful in practice.

\subsection{Nicol\'as--Saor\'in's construction}
Following the proof of \cite[Theorem 4]{NicolasSaorin09}, we consider the triangle
\begin{align}
\xymatrix@R=0.9pc{
j_!j^*(A)\ar@{=}[d]\ar[r] & A\ar[r] & i_*i^*(A)\ar[r] & \Sigma A.\\
Ae\lten_{eAe}eA
}
\end{align}
Let $Y=i_*i^*(A)$. We may assume that $Y$ is $\ch$-injective over $A$. Let $C=\cEnd_A(Y)$. Then $Y$ becomes a dg $C$-$A$-bimodule, in other words, a dg $C^{op}\ten A$-module. Let $Y\rightarrow Y'$ be an $\ch$-injective resolution of $Y$ as a dg $C^{op}\ten A$-module and let $B'=\cEnd_{C^{op}}(Y')^{op}$. Then the dg $C^{op}\ten A$-module structure on $Y'$ yields a homomorphism of dg algebras $f'=A\rightarrow B'$. Then $B=\sigma^{\leq 0}B'$ and $f=\sigma^{\leq 0}f'$ satisfy the desired properties in Proposition~\ref{p:recollement-from-projective-general-case}. We point out that $Y'$ is a quasi-equivalence from $C$ to $B$.

\subsection{Drinfeld's construction}

Define a dg $k$-category $\ca$ with two objects $\epsilon$ and $\gamma$: the morphism spaces are 
\begin{align*}
\Hom_\ca(\epsilon,\epsilon)&=eAe,\\
\Hom_\ca(\epsilon,\gamma)&=(1-e)Ae,\\ 
\Hom_\ca(\gamma,\epsilon)&=eA(1-e),\\ 
\Hom_\ca(\gamma,\gamma)&=(1-e)A(1-e),
\end{align*} 
and the composition of morphisms is induced from the multiplication of $A$.

Let $\ca'$ be the dg subcategory of $\ca$ consisting of the object $\epsilon$. Following \cite[Section 3.1]{Drinfeld04}, we form the dg quotient $\ca/\ca'$ by formally adjoining a morphism $x\colon \epsilon\to\epsilon$ of degree $-1$ such that $d(x)=\id_\epsilon=e$. Then  by \cite[Theorem 3.4]{Drinfeld04} the dg algebra $B=\End_{\ca/\ca'}(\gamma)$ and the inclusion $f\colon A\to B$ satisfy the desired properties in Proposition~\ref{p:recollement-from-projective-general-case}. As a complex, $B$ has the form
\begin{align*}
\xymatrix{
\ldots\ar[r] & Ae\ten (eAe)^{\ten p}\ten eA\ar[r]^(0.72){d^{-p-1}} & \ldots\ar[r] & Ae\ten eAe\ten eA\ar[r]^(0.58){d^{-2}} & Ae\ten eA\ar[r]^(0.62){d^{-1}} & A
},
\end{align*}
where the rightmost term is in degree $0$ and the term $Ae\ten (eAe)^{\ten p}\ten eA$ is in degree $-p-1$. The differentials are given by
\begin{align*}
d^{-p-1}(a_0\ten a_1\ten\cdots \ten a_p\ten a_{p+1})=\sum_{i=0}^{p+1}(-1)^i a_0\ten\cdots\ten a_ia_{i+1}\ten\cdots\ten a_{p+1}.
\end{align*}
The multiplication of $B$ is induced from that of $A$:
\begin{align*}
(a_0\ten \cdots\ten a_{p+1})(b_0\ten\cdots\ten b_{q+1})=a_0\ten\cdots \ten a_p\ten a_{p+1}b_0\ten b_1\ten\cdots\ten b_{q+1}.
\end{align*}

\subsection{Special case: $A$ is quasi-isomorphic to a dg quiver algebra}
\label{ss:construction-of-B-case:non-complete-cofibrant-model}
Assume that $k$ is a field. 

\begin{theorem}\label{t:construction-with-noncomplete-cofibrant-model} Keep the notation and assumptions in Proposition~\ref{p:recollement-from-projective-general-case}. If there is a quasi-isomorphism $\rho \colon  \tilde{A}\rightarrow A$ of dg algebras, where $\tilde{A}=(kQ,d)$ is a dg quiver algebra, such that $e$ is the image under $\rho$
of a sum $\tilde{e}$ of some trivial paths of $Q$, then $B$ is quasi-equivalent to $\tilde{A}/\tilde{A}\tilde{e}\tilde{A}$.
\end{theorem}

We need an auxiliary result. Let $A=(kQ,d)$ be a dg quiver algebra. 
 Let $Q_0'$ be a subset of $Q_0$ and let $e$ be the sum of the trivial paths at $i\in Q_0'$. The following lemma
 generalises~\cite[Lemma 7.2]{Keller11}.

\begin{lemma}\label{l:recollement-from-projective-free-case} 
The following diagram is a recollement
\begin{align}
\begin{xy}
\SelectTips{cm}{}
\xymatrix{\cd(A/AeA)\ar[rr]|(.55){\, \, i_*=i_! \, \,}&&\cd(A)\ar[rr]|{\, \, j^!=j^* \, \,} \ar@/^15pt/[]!<0ex, 0ex>;[ll]!<3ex, 0ex>|{\, \, i^! \, \, }\ar@/_15pt/[]!<0ex, 0ex>;[ll]!<3ex, 0ex>|{\, \, i^*\, \,}&&\cd(eAe)\ar@/^15pt/[ll]|{\, \, j_* \, \,}\ar@/_15pt/[ll]|{\, \, j_! \, \,}}
\end{xy},
\end{align}
where the respective triangle functors are explicitely given as follows
\[\begin{array}{ll}
i^*=?\lten_A A/AeA, & j_!=?\lten_{eAe} eA,\\
i_*=\RHom_{A/AeA}(A/AeA,?), &
j^!=\RHom_{A}(eA,?),\\
i_!=?\lten_{A/AeA}A/AeA, & j^*=?\lten_A Ae,\\
i^!=\RHom_A(A/AeA,?),& j_*=\RHom_{eAe}(Ae,?).
\end{array}\]
\end{lemma}
\begin{proof} We will show that $\Loc(eA)^\perp=\Loc(A/AeA)$ and $i_*$ is fully faithful, which imply that  the TTF triple $(\Tria(eA),\Tria(eA)^\perp, (\Tria(eA)^\perp)^\perp)$ corresponds to the desired recollement under the correspondences in~\cite[Theorem 5]{NicolasSaorin09}.

For the inclusion $\Loc(A/AeA)\subseteq \Loc(eA)^\perp$, it suffices to show $A/AeA\in\Loc(eA)^\perp$, which follows from
\[
\Hom_{\cd(A)}(eA,\Sigma^p A/AeA)=H^p\cHom_A(eA,A/AeA)=H^p((A/AeA)e)=0.
\]
For the inclusion $\Loc(eA)^\perp\subseteq \Loc(A/AeA)$, consider the triangle
\[
\xymatrix{
AeA\ar[r] & A\ar[r] & A/AeA\ar[r] &\Sigma AeA.
}
\]
This yields a triangle for any $M\in\Loc(eA)^\perp$
\[
\xymatrix{
M\lten_A AeA\ar[r] & M\ar[r] & M\lten_A A/AeA \ar[r] & \Sigma M\lten_A AeA.
}
\]
Since $AeA=Ae\ten_{eAe}eA$, it follows that $M\lten_A AeA\cong M\ten_A Ae\ten_{eAe} eA$ belongs to $\Loc(eA)$. Therefore the middle morphism in the above triangle must be an isomorphism, showing that $M$ belongs to $\Loc(A/AeA)$.

Finally,  $AeA\lten_A A/AeA=Ae\ten_{eAe} eA\ten_A A/AeA=0$, so the condition (4) in \cite[Lemma 4]{NicolasSaorin09} is satisfied, implying that $i_*$ is fully faithful. 
\end{proof}

Now Theorem~\ref{t:construction-with-noncomplete-cofibrant-model} is an immediate consequence of Lemmas~\ref{lem:B-in-homotopy-category-of-dg-algebras} and~\ref{l:recollement-from-projective-free-case}.

\subsection{An example}

We provide an example to illustrate some results of this section and Section~\ref{s:relative-singularity-category}.

\begin{example}
Let $A$ be as in Example~\ref{ex:s(2,2)} and take $e=e_1$, the trivial path at the vertex $1$. Note that $A$ is 5-dimensional over the field $k$. Applying Theorem~\ref{t:construction-with-noncomplete-cofibrant-model}, we obtain that there is a recollement of the form
\begin{align*}
\begin{xy}
\SelectTips{cm}{}
\xymatrix{\cd(k[t])\ar[rr]|(.55){\, \, i_*=i_! \, \,}&&\cd(A)\ar[rr]|{\, \, j^!=j^* \, \,} \ar@/^15pt/[]!<0ex, 0ex>;[ll]!<3ex, 0ex>|{\, \, i^! \, \, }\ar@/_15pt/[]!<0ex, 0ex>;[ll]!<3ex, 0ex>|{\, \, i^*\, \,}&&\cd(k[x]/x^2)\ar@/^15pt/[ll]|{\, \, j_* \, \,}\ar@/_15pt/[ll]|{\, \, j_! \, \,}}
\end{xy},
\end{align*}
where $\deg(x)=0$, $\deg(t)=-1$ and $k[t]$ is a dg algebra with trivial differential. Since $\cd_{sg}(k[x]/x^2)\cong \underline{\mod}k[x]/x^2$ is idempotent complete, it follows from Theorem~\ref{t:singularity-category-vs-relative-singularity-category} that $\ch^b(\proj A)/\thick(eA)$ is idempotent complete. Therefore, by Corollary~\ref{c:restriction-and-induction}, there is a triangle equivalence
\begin{align*}
\ch^b(\proj A)/\thick(eA)\stackrel{\simeq}{\longrightarrow}\per(k[t]).
\end{align*}
The Auslander--Reiten quiver of $\ch^b(\proj A)$ is
\[
{\scriptsize
\begin{xy} 0;<0.35pt,0pt>:<0pt,-0.35pt>::
(0,50) *+{} ="0",  (0,450) *+{}="1", 
(50,0) *+{\Sigma^{-1}P_1} ="2", (50,100) *+{} ="3", (50,400) *+{} ="4", (50,500) *+{S_2} ="5", 
(100,50) *+{\circ} ="6", (100,150) *+{} ="7", (100,350) *+{} ="8", (100,450) *+{\circ} ="9",
(150,0) *+{P_1} ="10", (150,100) *+{\circ} ="11", (150,200) *+{} ="12", (150,300) *+{} ="13", (150,400) *+{\circ} ="14", (150,500) *+{\Sigma^{-1}S_2} ="15",
(200,50) *+{\circ} ="16", (200,150) *+{\circ} ="17",  (200,350) *+{\circ} ="18", (200,450) *+{\circ} ="19",
(250,0) *+{\Sigma P_1} ="20", (250,100) *+{\circ} ="21", (250,200) *+{*} ="22", (250,300) *+{*} ="23", (250,400) *+{\circ} ="24", (250,500) *+{\circ} ="25",
(300,50) *+{\circ} ="26", (300,150) *+{*} ="27",  (300, 250) *+{\circ} ="28", (300,350) *+{*} ="29", (300,450) *+{\circ} ="30",
(350,0) *+{\circ} ="31", (350,100) *+{*} ="32", (350,200) *+{I_2} ="33", (350,300) *+{\circ} ="34", (350,400) *+{*} ="35", (350,500) *+{\circ} ="36", 
(400,50) *+{*} ="37", (400,150) *+{\circ} ="38",  (400, 250) *+{S_1} ="39", (400,350) *+{\circ} ="40", (400,450) *+{*} ="41",
(450,100) *+{\circ} ="42", (450,200) *+{\Sigma P_2} ="43", (450,300) *+{P_2} ="44", (450,400) *+{\circ} ="45", 
(500,50) *+{*} ="46", (500,150) *+{\circ} ="47",  (500, 250) *+{\circ} ="48", (500,350) *+{\circ} ="49", (500,450) *+{*} ="50",
(550,0) *+{\circ} ="51", (550,100) *+{*} ="52", (550,200) *+{\circ} ="53", (550,300) *+{\circ} ="54", (550,400) *+{*} ="55", (550,500) *+{\Sigma^{-1}P_1} ="56", 
(600,50) *+{\circ} ="57", (600,150) *+{*} ="58",  (600, 250) *+{\circ} ="59", (600,350) *+{*} ="60", (600,450) *+{\circ} ="61",
(650,0) *+{\Sigma S_2} ="62", (650,100) *+{\circ} ="63", (650,200) *+{*} ="64", (650,300) *+{*} ="65", (650,400) *+{\circ} ="66", (650,500) *+{P_1} ="67",
(700,50) *+{\circ} ="68", (700,150) *+{\circ} ="69",  (700,350) *+{\circ} ="70", (700,450) *+{\circ} ="71",
(750,0) *+{S_2} ="72", (750,100) *+{\circ} ="73", (750,200) *+{} ="74", (750,300) *+{} ="75", (750,400) *+{\circ} ="76", (750,500) *+{\Sigma P_1} ="77",
(800,50) *+{\circ} ="78", (800,150) *+{} ="79", (800,350) *+{} ="80", (800,450) *+{\circ} ="81",
(850,0) *+{\Sigma^{-1}S_2} ="82", (850,100) *+{} ="83", (850,400) *+{} ="84", (850,500) *+{\circ} ="85", 
(900,50) *+{} ="86",  (900,450) *+{}="87", 
"0", {\ar@{.}"2"}, 
"1", {\ar@{.}"5"}, 
"2", {\ar"6"}, 
"3", {\ar@{.}"6"},
"4", {\ar@{.}"9"}, 
"5", {\ar"9"},  
"6", {\ar"10"}, "6", {\ar"11"},
"7", {\ar@{.}"11"}, 
"8", {\ar@{.}"14"},
"9", {\ar"14"}, "9", {\ar"15"},
"10", {\ar"16"}, 
"11", {\ar"16"}, "11", {\ar"17"}, 
"12", {\ar@{.}"17"}, 
"13", {\ar@{.}"18"},
"14", {\ar"18"}, "14", {\ar"19"}, 
"15", {\ar"19"},
"16", {\ar"20"}, "16", {\ar"21"},
"17", {\ar"21"}, "17", {\ar@{.}"22"},
"18", {\ar@{.}"23"}, "18", {\ar"24"}, 
"19", {\ar"24"}, "19", {\ar"25"},
"20", {\ar"26"},
"21", {\ar"26"}, "21", {\ar@{.}"27"},
"22", {\ar@{.}"28"},
"23", {\ar@{.}"28"},
"24", {\ar@{.}"29"}, "24", {\ar"30"},
"25", {\ar"30"},
"26", {\ar"31"}, "26", {\ar@{.}"32"},
"27", {\ar@{.}"33"},
"28", {\ar"33"}, "28", {\ar"34"},
"29", {\ar@{.}"34"},
"30", {\ar@{.}"35"}, "30", {\ar"36"},
"31", {\ar@{.}"37"}, 
"32", {\ar@{.}"38"},
"33", {\ar"38"}, "33", {\ar"39"},
"34", {\ar"39"}, "34", {\ar"40"},
"35", {\ar@{.}"40"},
"36", {\ar@{.}"41"},
"37", {\ar@{.}"42"},
"38", {\ar"42"}, "38", {\ar"43"},
"39", {\ar"43"}, "39", {\ar"44"},
"40", {\ar"44"}, "40", {\ar"45"},
"41", {\ar@{.}"45"},
"42", {\ar@{.}"46"}, "42", {\ar"47"},
"43", {\ar"47"}, "43", {\ar"48"},
"44", {\ar"48"}, "44", {\ar"49"},
"45", {\ar"49"}, "45", {\ar@{.}"50"},
"46", {\ar@{.}"51"},
"47", {\ar@{.}"52"}, "47", {\ar"53"},
"48", {\ar"53"}, "48", {\ar"54"},
"49", {\ar"54"}, "49", {\ar@{.}"55"},
"50", {\ar@{.}"56"},
"51", {\ar"57"},
"52", {\ar@{.}"57"},
"53", {\ar@{.}"58"}, "53", {\ar"59"},
"54", {\ar"59"}, "54", {\ar@{.}"60"},
"55", {\ar@{.}"61"},
"56", {\ar"61"},
"57", {\ar"62"}, "57", {\ar"63"},
"58", {\ar@{.}"63"}, 
"59", {\ar@{.}"64"}, "59", {\ar@{.}"65"},
"60", {\ar@{.}"66"},
"61", {\ar"66"}, "61", {\ar"67"},
"62", {\ar"68"},
"63", {\ar"68"}, "63", {\ar"69"},
"64", {\ar@{.}"69"},
"65", {\ar@{.}"70"},
"66", {\ar"70"}, "66", {\ar"71"},
"67", {\ar"71"},
"68", {\ar"72"}, "68", {\ar"73"},
"69", {\ar"73"}, "69", {\ar@{.}"74"},
"70", {\ar@{.}"75"}, "70", {\ar"76"},
"71", {\ar"76"}, "71", {\ar"77"},
"72", {\ar"78"},
"73", {\ar"78"}, "73", {\ar@{.}"79"},
"76", {\ar@{.}"80"}, "76", {\ar"81"},
"77", {\ar"81"},
"78", {\ar"82"}, "78", {\ar@{.}"83"},
"81", {\ar@{.}"84"}, "81", {\ar"85"},
"82", {\ar@{.}"86"},
"85", {\ar@{.}"87"},
\end{xy}
}
\]
where $P_1=eA$ and $P_2=(1-e)A$ are the two indecomposable projective $A$-modules, $S_1=P_1/\rad P_1$ and $S_2=P_2/\rad P_2=A/AeA$ are the two simple $A$-modules, and $I_1=P_1$ and $I_2=D(A(1-e))$ are the two indecomposable injective $A$-modules. The $*$'s are virtual connecting objects (actually they are limits and colimits of objects of $\cd^b(\mod A)$ and belong to $\cd(A)$). We point out that the left upper component is identified with the right lower component, and the left lower component is identified with the right upper component. So there are three connected components, which are related by morphisms in the infinite radical.

The functor $i^*\colon \ch^b(\proj A)\rightarrow \per(k[t])$ kills the component containing $P_1=eA$ and identifies an object in the $\mathbb{Z}A^\infty_\infty$ component with its right upper neighbour. Thus, we obtain the Auslander--Reiten quiver of $\per(k[t])$
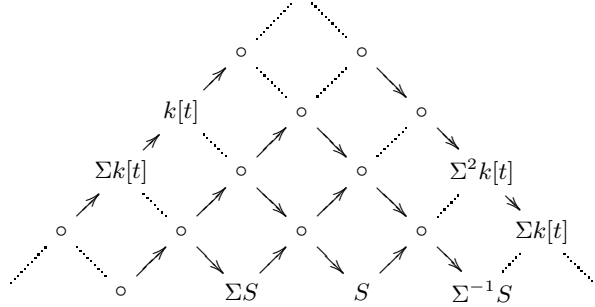
\[
{\scriptsize
\begin{xy} 0;<0.45pt,0pt>:<0pt,-0.45pt>::
(0,250) *+{} ="0",  
(50,200) *+{\circ} ="1", 
(100,150) *+{\Sigma k[t]} ="2", (100,250) *+{\circ} ="3",
(150,100) *+{k[t]} = "4", (150,200) *+{\circ} = "5",
(200,50) *+{\circ} ="6", (200,150) *+{\circ} = "7", (200,250) *+{\Sigma S} ="8",
(250,0) *+{}="9", (250,100) *+{\circ} ="10", (250,200) *+{\circ} ="11",
(300,50) *+{\circ} ="12", (300,150) *+{\circ} = "13", (300,250) *+{S} ="14",
(350,100) *+{\circ} = "15", (350,200) *+{\circ} = "16",
(400,150) *+{\Sigma^2 k[t]} ="17", (400,250) *+{\Sigma^{-1}S} ="18",
(450,200) *+{\Sigma k[t]} ="19",
(500,250) *+{} ="20",
"0", {\ar@{.}"1"},
"1", {\ar"2"}, {\ar@{.}"3"},
"2", {\ar"4"}, {\ar@{.}"5"},
"3", {\ar"5"},
"4", {\ar"6",\ar@{.}"7"},
"5", {\ar"7",\ar"8"},
"6", {\ar@{.}"9",\ar@{.}"10"},
"7", {\ar"10",\ar"11"},
"8", {\ar"11"},
"9", {\ar@{.}"12"},
"10", {\ar@{.}"12",\ar"13"},
"11", {\ar"13",\ar"14"},
"12", {\ar"15"},
"13", {\ar@{.}"15",\ar"16"},
"14", {\ar"16"},
"15", {\ar"17"},
"16", {\ar@{.}"17",\ar"18"},
"17", {\ar"19"},
"18", {\ar@{.}"19"},
"19", {\ar@{.}"20"},
\end{xy}
}
\]
where $S=k[t]/(t)$ is the simple dg $k[t]$-module concentrated in degree $0$, and the left and right rays are identified. We remark that this quiver has two connected components: the one containing $S$ has the shape $\mathbb{Z}A_\infty$ and the one containing $k[t]$ has shape $A^\infty_\infty$. The two components are related by morphisms in the infinite radical.

By Theorem~\ref{t:singularity-category-vs-relative-singularity-category}, there is a triangle functor
\begin{align*}
\per(k[t])\longrightarrow \cd_{sg}(k[x]/x^2)\simeq \underline{\mod}k[x]/x^2
\end{align*}
which is a quotient functor and has kernel $\cd_{fd}(k[t])$. On the Auslander--Reiten quiver, this functor kills the component containing $S$ and identifies all objects in the component containing $k[t]$.
\end{example}

\section{Algebras with radical square zero}

In this section we recover a structural result obtained in \cite{ChenYang15} on singularity categories of algebras with radical square zero by using results in Sections~\ref{s:universal-localisation-of-graded-algebras},~\ref{s:relative-singularity-category} and~\ref{s:dg-model}.

\smallskip
Assume that $k$ is a field.
Let $Q$ be a finite quiver and
let $R$ be the corresponding algebra with radical square zero. 
Assume $Q_0=\{1,\ldots,n\}$ and let $P_i$ (respectively, $S_i$) denote the indecomposable projective module (respectively, simple module) corresponding to the vertex $i\in Q_0$. It is known by \cite[Corollary 3.2]{ChenXW11a} that $\cd_{sg}(R)$ is idempotent complete.
More structural results on $\cd_{sg}(R)$ were obtained in \cite{ChenXW11a,ChenYang15}. We will provide an alternative approach to \cite[Theorem 6.1]{ChenYang15}, which identifies $\cd_{sg}(R)$ with the perfect derived category of the Leavitt path algebra of $Q^{op}$.

We assume that $Q$ has no sources. So there are no simple projective modules. 
Take $M=R\oplus S_1\oplus\ldots\oplus S_n$, and let $A=\End_R(M)$. By a result of Auslander \cite[III.3]{Auslander71}, we know that the global dimension of $A$ is at most $2$. Let $Q^A$ be the quiver of $A$. Then
\begin{itemize}
\item[--] the vertices of $Q^A$ are $i,~i'$, where $i\in Q_0$; 
\item[--] 
the following are all the arrows of $Q^A$:
\begin{itemize} 
\item[$\cdot$] for each $i\in Q_0$, there is an arrow $c_i\colon i\to i'$ (corresponding to the projection $P_i\to S_i$), 
\item[$\cdot$] for each $\alpha\in Q_1$, there is an arrow $a_\alpha\colon s(\alpha)'\to t(\alpha)$ (corresponding to the embedding $S_{s(\alpha)}\to P_{t(\alpha)}$).
\end{itemize}
\end{itemize}
For each $\alpha\in Q_1$, there is a relation $c_{t(\alpha)}a_{\alpha}$ in $A$. By a dimension comparison it is easy to show that these are all the relations. 

By Lemma~\ref{lem:non-complete-cofibrant-model}, $A$ is quasi-isomorphic to the dg quiver algebra $\tilde{A}=(k\tilde{Q},d)$, where $\tilde{Q}$ is the graded quiver which has the same vertices as $Q$ and which has three types of arrows
\begin{itemize}
\item[$\cdot$] for each $i\in Q_0$, an arrow $c_i\colon i\to i'$ of degree $0$;
\item[$\cdot$] for each $\alpha\in Q_1$, an arrow $a_\alpha\colon s(\alpha)'\to t(\alpha)$ of degree $0$;
\item[$\cdot$] for each $\alpha\in Q_1$, an arrow $\alpha'\colon s(\alpha)'\to t(\alpha)'$ of degree $-1$.
\end{itemize}
The differential $d$ takes $\alpha'$ to $c_{t(\alpha)}a_\alpha$ and takes other arrows to $0$.

Let $e=\sum_{i\in Q_0} e_i$, and let $Q'$ be the graded quiver obtained from $Q$ by putting all arrows in degree $-1$. Then  $B=\tilde{A}/\tilde{A}e\tilde{A}$ is  the graded path algebra $kQ'$, considered as a dg algebra with trivial differential.  

\begin{corollary}
\label{cor:singularity-cat-is-AGK-cat-rad2=0}
There is a triangle equivalence 
\[
\cd_{sg}(R)\longrightarrow \per(B)/\cd_{fd}(B).
\]
\end{corollary}
\begin{proof}
This follows from Theorems~\ref{t:singularity-category-vs-relative-singularity-category} and \ref{t:construction-with-noncomplete-cofibrant-model} because $\cd_{sg}(R)$ is idempotent complete.
\end{proof}

Below we use the theory of universal localisation of graded algebras to study the structure of $\per(B)/\cd_{fd}(B)$, compare \cite[Section 5.8]{Smith12}. There is a set of simple graded $B$-modules $S_1^B,\ldots,S_n^B$, one for each vertex of $Q$. For $i\in Q_0$, 
\[
f_i\colon  \bigoplus_{\alpha'\in Q'_1\colon  t(\alpha')=i}e_{s(\alpha')}B\langle 1\rangle \stackrel{(\alpha'\cdot)}{\longrightarrow} e_i B
\]
is a minimal projective resolution of $S_i^B$ as graded $B$-modules. Put $f=\bigoplus_{i\in Q_0} f_i$. Then the universal localisation $B_f$ of $B$ at $f$ is obtained from $B$ by adjoining new arrows $\alpha'^*\colon t(\alpha')\to s(\alpha')$ of degree $1$ for each $\alpha'\in Q'_1$ subject to the following relations:
\begin{itemize}
\item[$\cdot$] $\sum_{\alpha'\in Q'_1\colon  t(\alpha')=i}\alpha'\alpha'^*=e_i$,
\item[$\cdot$] $\alpha'^*\alpha'=e_{s(\alpha')}$ for each $\alpha'\in Q'_1$,
\item[$\cdot$] $\alpha'^*\beta=0$ for $\alpha',\beta'\in Q'_1$ with $\alpha'\neq \beta'$.
\end{itemize}
This is exactly the Leavitt path algebra $L(Q^{op})$ of the quiver $Q^{op}$ (see \cite{AbramsAranda05} for the definition of Leavitt path algebras).

\begin{theorem}[{\cite[Theorem 6.1]{ChenYang15}}]
\label{thm:rad^2=0-vs-Leavitt}
There is a triangle equivalence
\[
\cd_{sg}(R)\simeq \per(L(Q^{op})).
\]
\end{theorem}
\begin{proof}
By Proposition~\ref{prop:standard-t-str}, $\cd_{fd}(B)=\thick(S_i^B|i\in Q_0)$. So by Theorem~\ref{thm:localisation-of-graded-alg}, there is a triangle equivalence
\[
\per(B)/\cd_{fd}(B)\simeq\per(B_f)=\per(L(Q^{op})).
\] 
The desired equivalence then follows from Corollary~\ref{cor:singularity-cat-is-AGK-cat-rad2=0}.
\end{proof}

\section{Non-commutative deformations of Kleinian singularities}

In this section we will use our results in Sections~\ref{s:relative-singularity-category} and \ref{s:dg-model} to recover a result of Crawford~\cite{Crawford16} stating that the singularity category of a deformed Kleinian singularity in the sense of Crawley-Boevey--Holland~\cite{CrawleyBoeveyHolland98} is equivalent to the direct sum of singularity categories of certain associated Kleinian singularities.

\smallskip
Assume that $k$ is a field.

\subsection{Deformed derived preprojective algebras}

Let $Q$ be a finite quiver without oriented cycles. Assume $Q_0=\{1,\ldots,n\}$ and let $\lambda=(\lambda_1,\ldots,\lambda_n)\in k^n$.

The \emph{deformed derived preprojective algebra} $\mathbf{\Pi}^\lambda(Q)$ \cite{Keller11} is defined as the dg quiver algebra $(k\tilde{Q},d)$, where $\tilde{Q}$ is the graded quiver whose vertices are the same as $Q$ and whose arrows are
\begin{itemize}
\item[-] arrows of $Q$ (they are in degree $0$),
\item[-] $a^*\colon t(a)\to s(a)$ in degree $0$ for each $a\in Q_1$,
\item[-] $t_i\colon i\to i$ in degree $-1$ for each $i\in Q_0$. 
\end{itemize}
The differential $d$ takes the following values on arrows
\begin{itemize}
\item[-] $d(a)=0$ for $a\in Q_1$,
\item[-] $d(a^*)=0$ for $a\in Q_1$,
\item[-] $d(t_i)=e_i\sum_{a\in Q_1} (aa^*-a^*a) e_i-\lambda_ie_i$ for $i\in Q_0$.
\end{itemize}

The zeroth cohomology $\Pi^\lambda(Q)$ of $\mathbf{\Pi}^\lambda(Q)$ is the \emph{deformed preprojective algebra} of $Q$ in the sense of \cite{CrawleyBoeveyHolland98}. Precisely,
\[
\Pi^\lambda(Q)=k\bar{Q}/(e_i\sum_{a\in Q_1} (aa^*-a^*a) e_i-\lambda_ie_i|i\in Q_0),
\]
where $\bar{Q}$ is the double quiver of $Q$, obtained from $Q'$ by removing the $t_i$'s. 
The dg algebra $\mathbf{\Pi}(Q):=\mathbf{\Pi}^0(Q)$ is the \emph{derived preprojective algebra} of $Q$ and the algebra $\Pi(Q):=\Pi^0(Q)$ is the \emph{preprojective algebra} of $Q$.

\begin{lemma}\label{lem:acyclicity-of-derived-preprojective-algebra}
Let $Q$ be a connected quiver which is not Dynkin. Then the projection $\mathbf{\Pi}^\lambda(Q)\to \Pi^\lambda(Q)$ is a quasi-isomorphism.
\end{lemma}
\begin{proof}
It is known that $\Pi(Q)$ is Koszul (\cite[Theorem 1.9]{MartinezVilla96}) and of global dimension $2$ (\cite[Proposition 4.2]{BaerGeigleLenzing87}). So it follows from Lemma~\ref{lem:non-complete-cofibrant-model} and Remark~\ref{rem:cofibrant-model-for-Koszul-algebra} that the desired result holds for $\lambda=0$. In other words, $\mathbf{\Pi}(Q)$ has cohomology concentrated in degree $0$ (see also \cite[Section 4.2]{Keller10}).

The rest of the proof is borrowed from \cite[the proof of Corollary 5.4.4]{Ginzburg06}. Define an Adams grading $|\cdot|_{\mathrm{a}}$ on $k\tilde{Q}$ by setting $|a|_{\mathrm{a}}=-1=|a^*|_{\mathrm{a}}$ for $a\in Q_1$ and $|t_i|_{\mathrm{a}}=-2$ for $i\in Q_0$. Consider the filtration of $kQ'$
\[
0=F^{1}\subset F^0\subset F^{-1}\subset\ldots\subset F^{p+1}\subset F^{p}\subset\ldots\subset k\tilde{Q},
\]
where $F^p$ is the $k$-span of paths of $\tilde{Q}$ of Adams degree $\geq p$. Then $F^p$ is closed under the differential $d$, and in this way $\mathbf{\Pi}^\lambda(Q)$ becomes a filtered dg algebra. It is clear that the subquotient $F^p/F^{p+1}$ is the $k$-span of paths of $Q'$ of Adams degree $p$, and the definition of $d$ shows that the associated graded dg algebra is $\mathbf{\Pi}(Q)^{gr}$, the dg algebra $\mathbf{\Pi}(Q)$ with an extra grading given by the Adams grading.
Consider the two spectral sequences associated to the above filtrations for $\mathbf{\Pi}^\lambda(Q)$ and for $\mathbf{\Pi}(Q)$, which converge to $H^*\mathbf{\Pi}^\lambda(Q)$ and $H^*\mathbf{\Pi}(Q)$, respectively. 
These two spectral sequences have the same $E_0$ pages, with $E_0^{pq}=(\mathbf{\Pi}(Q)^{gr})^{p,p+q}$ (where the Adams grading is set as the first grading) and the differential is the restriction of the differential of $\mathbf{\Pi}(Q)^{gr}$, and therefore the terms of the two $E_1$ pages are also the same, with $E_1^{pq}=H^{p,p+q}\mathbf{\Pi}(Q)^{gr}$. Since $\mathbf{\Pi}(Q)$ has cohomology concentrated in degree $0$, it follows that $E_1^{pq}$ vanishes unless $p=-q$, and hence the two spectral sequences are the same.  
As a consequence, the cohomology of $\mathbf{\Pi}^\lambda(Q)$ is concentrated in degree $0$. This finishes the proof.
\end{proof}

Fix a vertex $i\in Q_0$ and let $Q'$ be the quiver obtained from $Q$ by deleting the vertex $i$ and $\lambda'$ be the tuple obtained from $\lambda$ by removing the $i$-th entry. The following lemma is clear from the definition of $\mathbf{\Pi}^\lambda(Q)$.

\begin{lemma}\label{lem:deleting-a-vetex-preprojective-algebra}
$\mathbf{\Pi}^\lambda(Q)/\mathbf{\Pi}^\lambda(Q)e_i\mathbf{\Pi}^\lambda(Q)=\mathbf{\Pi}^{\lambda'}(Q')$.
\end{lemma}

\subsection{Dynkin quivers}
Assume that the characteristic of $k$ is $0$.
Let $Q$ be a Dynkin quiver. Assume $Q_0=\{1,\ldots,n\}$ and let $\lambda=(\lambda_1,\ldots,\lambda_n)\in k^n$. Let $I_\lambda$ be the set of vertices $i$ of $Q$ such that $\lambda_i=0$ and let $Q_\lambda$ be the full subquiver of $Q$ with vertex set $I_\lambda$. By Lemma~\ref{lem:deleting-a-vetex-preprojective-algebra}, there is a surjective dg algebra homomorphism $\mathbf{\Pi}^\lambda(Q)\to\mathbf{\Pi}(Q_\lambda)$ with kernel the ideal generated by the trivial paths $e_i,~i\not\in I_\lambda$.

\begin{lemma}\label{lem:reduction-for-dominant-weight}
Assume that $\lambda$ is dominant in the sense of \cite[Section 7]{CrawleyBoeveyHolland98}. Then the surjective dg algebra homomorphism $\mathbf{\Pi}^\lambda(Q)\to\mathbf{\Pi}(Q_\lambda)$ is a quasi-isomorphism.
\end{lemma}
\begin{proof}
There is a surjective algebra homomorphism $\Pi^\lambda(Q)\to\Pi(Q_\lambda)$ with kernel the ideal generated by $e_i,~i\not\in I_\lambda$. Since both algebras are finite-dimensional (\cite[Lemma 2.4]{CrawleyBoeveyHolland98}), it follows from \cite[Lemma 7.1(1)]{CrawleyBoeveyHolland98} that this surjective algebra homomorphism is an isomorphism. As a consequence, for each $i\not\in I_\lambda$, there exists $x_i\in e_ikQ' e_i$ such that $d(x_i)=e_i$. Now applying Lemma~\ref{lem:deleting-a-contractible-vertex} we obtain the desired result.
\end{proof}

\subsection{Singularity categories} Assume that the characteristic of $k$ is $0$.
Let $Q$ be an Euclidean quiver with $0$ an extending vertex and let $Q'$ be the Dynkin quiver obtained from $Q$ by deleting $0$. Let $\lambda=(\lambda_0,\lambda_1,\ldots,\lambda_n)\in k^{n+1}$ and $\lambda'=(\lambda_1,\ldots,\lambda_n)$. $\lambda$ is said to be \emph{quasi-dominant} if $\lambda'$ is dominant.

Put $R=e_0\Pi^{\lambda}(Q)e_0$. If $\lambda=0$, then $R$ is the Kleinian singularity of type $Q'$. According to \cite[Lemma 2.14]{Crawford16}, in the isomorphism class of $R$ there is one with $\lambda$ quasi-dominant. Write $Q'_{\lambda'}=Q^{(1)}\cup \ldots\cup Q^{(s)}$ as the disjoint union of Dynkin quivers. Let $R^{(1)},\ldots,R^{(s)}$ be the corresponding Kleinian singularities.

\begin{theorem}[{\cite[Theorem 4.4]{Crawford16}}]
\label{thm:sing-cat-of-deformed-preproj-alg}
Assume that $\lambda$ is quasi-dominant.
Then there is a triangle equivalence 
\[
\cd_{sg}(R)\simeq \cd_{sg}(R^{(1)})\oplus\ldots\oplus\cd_{sg}(R^{(s)}).
\]
\end{theorem}
\begin{proof}
Let $e=e_0$, $A=\mathbf{\Pi}^\lambda(Q)$ and $B$ be the dg algebra obtained in Proposition~\ref{p:recollement-from-projective-general-case}. Since $A$ has global dimension at most $2$  by \cite[Theorem 1.5 and Corollary 3.5]{CrawleyBoeveyHolland98}, it follows from Theorem~\ref{t:singularity-category-vs-relative-singularity-category} that there is a triangle equivalence $\cd_{sg}(R)\simeq \per(B)/\cd_{fg}(B)$ up to direct summands, which is in fact a triangle equivalence as $\cd_{sg}(R)$ is idempotent complete (\cite[Corollary 3.7]{Crawford16}). By Lemma~\ref{lem:acyclicity-of-derived-preprojective-algebra}, the projection $\mathbf{\Pi}^\lambda(Q)\to \Pi^\lambda(Q)$ is a quasi-isomorphism of dg algebras. So by Theorem~\ref{t:construction-with-noncomplete-cofibrant-model} that we can take $B=\mathbf{\Pi}^\lambda(Q)/\mathbf{\Pi}^\lambda(Q)e\mathbf{\Pi}^\lambda(Q)$, which is $\mathbf{\Pi}^{\lambda'}(Q')$ by Lemma~\ref{lem:deleting-a-vetex-preprojective-algebra}. So there is a triangle equivalence
\[
\cd_{sg}(R)\simeq \per(\mathbf{\Pi}^{\lambda'}(Q'))/\cd_{fd}(\mathbf{\Pi}^{\lambda'}(Q')).
\]
By Lemma~\ref{lem:reduction-for-dominant-weight}, there are triangle equivalences
\begin{align*}
\per&(\mathbf{\Pi}^{\lambda'}(Q'))/\cd_{fd}(\mathbf{\Pi}^{\lambda'}(Q'))
\simeq 
\per(\mathbf{\Pi}(Q'_{\lambda'}))/\cd_{fd}(\mathbf{\Pi}(Q'_{\lambda'}))\\
&\simeq
\per(\mathbf{\Pi}(Q^{(1)}))/\cd_{fd}(\mathbf{\Pi}(Q^{(1)}))\oplus\ldots\oplus \per(\mathbf{\Pi}(Q^{(s)}))/\cd_{fd}(\mathbf{\Pi}(Q^{(s)})).
\end{align*}
It follows that there is a triangle equivalence
\[
\cd_{sg}(R)\simeq \cd_{sg}(R^{(1)})\oplus\ldots\oplus\cd_{sg}(R^{(s)}),
\]
as desired.
\end{proof}

\section{3D quotient singularities and toric threefolds}
In this section we show that the singularity categories of  3D Gorenstein quotient singularities 
and Gorenstein affine toric threefolds are (up to direct summands) triangle equivalent to small cluster categories of certain quivers with potential. This result extends earlier results in \cite{deVolcseyVandenBergh16} and \cite{AmiotIyamaReiten15}, which are for isolated singularities.

\smallskip
Assume that $k$ is a field.

\subsection{Quivers with potential}
Let $Q$ be a finite quiver. We denote by $\mathrm{Pot}(Q)$ the (finite) linear combinations of cycles of $Q$ and call elements of $\mathrm{Pot}(Q)$ \emph{finite potentials} of $Q$.  In this section we  deal with finite potentials only and we will simply call them potentials. 
For an arrow $a$ of $Q$, we define $ \del_a \colon  \mathrm{Pot}(Q) \to {kQ} $, the {\em cyclic derivative with respect to $a$}, 
as the unique $k$-linear map which takes a cycle $c$ to the sum
$ \sum_{c=u a v} vu $ taken over all decompositions of the cycle $c$
(where $u$ and $v$ are possibly trivial paths). 

Let $(Q,W)$ be a quiver with potential. The \emph{Ginzburg dg algebra ${\Gamma}(Q,W)$} is
constructed as the dg quiver algebra $(\widehat{k\tilde{Q}},d)$ \cite{Ginzburg06}: The graded quiver $\tilde{Q}$ has the same vertices as $Q$ and its arrows are
\begin{itemize}
\item the arrows of $Q$ (they all have degree~$0$),
\item an arrow $a^*\colon  j \to i$ of degree $-1$ for each arrow $a\colon i\to j$ of $Q$,
\item a loop $t_i \colon  i \to i$ of degree $-2$ for each vertex $i$
of $Q$.
\end{itemize}
The differential $d$ takes the
following values on the arrows of $\tilde{Q}$:
\begin{itemize}
\item $d(a)=0$ for each arrow $a$ of $Q$,
\item $d(a^*) = \del_a W$ for each arrow $a$ of $Q$,
\item $d(t_i) = e_i (\sum_{a} [a,a^*]) e_i$ for each vertex $i$ of $Q$, where
$e_i$ is the trivial path at $i$ and the sum runs over the set of
arrows of $Q$.
\end{itemize}
The \emph{Jacobian algebra} $J(Q,W)$ of $(Q,W)$ is defined as the $0$-th cohomology of $\Gamma(Q,W)$, that is,
\[
J(Q,W):=kQ/(\del_aW\colon a\in Q_1).
\]
Note that $\Gamma(Q,W)$ is concentrated in non-positive degrees, so there is a projection map $\Gamma(Q,W)\to J(Q,W)$, which is a surjective homomorphism of dg algebras.

\subsection{Cluster categories}
Let $(Q,W)$ be a quiver with potential. The \emph{cluster category} \cite{Amiot09, Plamondon11} of $(Q,W)$ is defined as the triangle quotient
\[\cc(Q,W):=\per(\Gamma(Q,W))/\cd_{fd}(\Gamma(Q,W)).\]

Assume that $J(Q,W)$ is right noetherian as a ring. Then $\cd_{fg}(\Gamma(Q,W))$ is a triangulated subcategory of $\cd(\Gamma(Q,W))$. Assume further that $\cd_{fg}(\Gamma(Q,W))\subseteq \per(\Gamma(Q,W))$. Define the \emph{small cluster category} of $(Q,W)$ as the triangle quotient
\[
\cc^s(Q,W):=\per(\Gamma(Q,W))/\cd_{fg}(\Gamma(Q,W)).
\]
If $J(Q,W)$ is finite-dimensional, then $\cd_{fg}(\Gamma(Q,W))=\cd_{fd}(\Gamma(Q,W))$, and $\cc^s(Q,W)=\cc(Q,W)$. In general, $\cd_{fd}(\Gamma(Q,W))$ is a proper subcategory of $\cd_{fg}(\Gamma(Q,W))$ and $\cc^s(Q,W)$ is a proper triangle quotient of $\cc(Q,W)$. 

\subsection{3-Calabi--Yau Jacobian algebras}

A dg $k$-algebra $A$ is said to be \emph{bimodule 3-Calabi--Yau} if it is homologically smooth and there is an isomorphism in $\cd(A^{op}\ten_k A)$
\[
\RHom_{A^{op}\ten_k A}(A,A^{op}\ten_k A)\cong \Sigma^{-3}A.
\]
In the original definition of Ginzburg in \cite{Ginzburg06}, the above isomorphism is required to be self-dual. But this turns out to be automatic, see \cite[Appendix 14]{VandenBergh15}.

\begin{theorem}[{\cite[Theorem 6.3]{Keller11}}]
The Ginzburg dg algebra $\Gamma(Q,W)$ of a quiver with potential $(Q,W)$ is bimodule $3$-Calabi--Yau.
\end{theorem}

A quiver with potential is said to be \emph{positively graded} if the quiver is graded with all arrows in positive degrees such that the potential is homogeneous. 

\begin{theorem}[{\cite[Corollary 5.4.3]{Ginzburg06}}] \label{thm:Jacobian-CY==>Ginzburg-formal}
Assume $k=\mathbb{C}$.
Let $(Q,W)$ be a positively graded quiver with potential. Then the following conditions are equivalent:
\begin{itemize}
\item[(i)] $J(Q,W)$ is bimodule 3-Calabi--Yau,
\item[(ii)] the projection $\Gamma(Q,W)\to J(Q,W)$ is a quasi-isomorphism.
\end{itemize}
\end{theorem}

\subsection{3D quotient singularities}
Assume $k=\mathbb{C}$.

Let $G\subseteq \mathrm{SL}_3(\mathbb{C})$ be a finite subgroup. Then $G$ naturally acts on $\mathbb{C}^3$ and on its coordinate ring $\mathbb{C}[x,y,z]$. Let $R=\mathbb{C}[x,y,z]^G$ be the subalgebra of $G$-invariant elements.

Let $Q$ be the McKay quiver of $G$ with $0$ the vertex corresponding to the trivial representation of $G$.

\begin{theorem}[{\cite[Theorem 4.4.6]{Ginzburg06}}]
There is a potential $W$ such that 
\begin{itemize}
\item[-] $J(Q,W)$ is bimodule 3-Calabi--Yau,
\item[-] $e_0 J(Q,W)e_0\cong R$,
\item[-] $J(Q,W)$ is isomorphic to the endomorphism algebra of a maximal Cohen--Macaulay $R$-module $M$ which has $R$ as a direct summand and $e_0=\mathrm{id}_R$ up to this isomorphism.
\end{itemize}
\end{theorem} 

We refer to \cite[Section 4.4]{Ginzburg06} for the detailed construction of $(Q,W)$. By construction $W$ is a linear combination of cycles of length $3$. So putting all arrows of $Q$ in degree $1$ makes $(Q,W)$ a positively graded quiver with potential. Therefore by Theorem~\ref{thm:Jacobian-CY==>Ginzburg-formal}, the projection $\Gamma(Q,W)\to J(Q,W)$ is a quasi-isomorphism of dg algebras. Let $(Q_G,W_G)$ be the quiver with potential obtained from $(Q,W)$ by deleting the vertex $0$.

\begin{theorem}\label{thm:singularity-category-of-quotient-singularity}
There is a triangle equivalence $\cd_{sg}(R)\simeq \cc^s(Q_G,W_G)$ up to direct summands. If $\cd_{sg}(R)$ is idempotent complete, then this is a triangle equivalence.
\end{theorem}
\begin{proof}
Let $A=J(Q,W)$, $e=e_0$ and $B$ be the dg algebra obtained in Proposition~\ref{p:recollement-from-projective-general-case}. Then $eAe\cong R$ and $A$ has global dimension $3$ by \cite[Corollary 5.3.3]{Ginzburg06}. So by Theorem~\ref{t:singularity-category-vs-relative-singularity-category}, there is a triangle equivalence $\cd_{sg}(R)\simeq \per(B)/\cd_{fg}(B)$ up to direct summands, which is a triangle equivalence if $\cd_{sg}(R)$ is idempotent complete. Since the projection $\Gamma(Q,W)\to J(Q,W)$ is a quasi-isomorphism of dg algebras, it follows by Theorem~\ref{t:construction-with-noncomplete-cofibrant-model} that we can take $B=\Gamma(Q,W)/\Gamma(Q,W)e_0\Gamma(Q,W)$, which is isomorphic to $\Gamma(Q_G,W_G)$.
\end{proof}

Assume that $G$ is a cyclic group with generator $g=\mathrm{diag}(\zeta^{a_1},\zeta^{a_2},\zeta^{a_3})$, where $\zeta$ is a primitive $n$-th root of unity and $a_1,a_2,a_3$ are integers satisfying $0<a_j<n$ and $(n,a_j)=1$ for $j=1,2,3$. According to \cite[Corollary 5.3]{AmiotIyamaReiten15}, there exists a finite-dimensional algebra $\underline{A}$ of global dimension at most $2$ such that $\cd_{sg}(R)$ is triangle equivalent to the generalised $2$-cluster category $\cc_2(\underline{A})$ of $\underline{A}$. By \cite[Theorem 6.12(a)]{Keller11}, there is a quiver with potential $(Q',W')$ such that $\cc_2(\underline{A})$, and hence $\cd_{sg}(R)$, is triangle equivalent to $\cc(Q',W')$. As a consequence of Theorem~\ref{thm:singularity-category-of-quotient-singularity}, we obtain such an equivalence for all isolated $SL_3(\mathbb{C})$-quotient singularities. The `complete' version of this result is a consequence of \cite[Proposition 1.2]{deVolcseyVandenBergh16}.

\begin{corollary} \label{cor:singularity-category-for-isolated-quotient-singularities}
Assume that $R$ has isolated singularity. 
Then there is a triangle equivalence $\cd_{sg}(R)\simeq \cc(Q_G,W_G)$.
\end{corollary}
\begin{proof}
Assume that $R$ has isolated singularity. Then $\cd_{sg}(R)$ is idempotent complete and Hom-finite by Theorem~\ref{thm:idempotent-completeness-and-Hom-finiteness-of-singularity-category}. Therefore by Theorem~\ref{thm:singularity-category-of-quotient-singularity} there is a triangle equivalence
\[
\cd_{sg}(R)\simeq \cc^s(Q_G,W_G).
\]
Moreover, $J(Q_G,W_G)=J(Q,W)/J(Q,W)e_0J(Q,W)$, the stable endomorphism algebra of $M$, is finite-dimensional, and $\cc^s(Q_G,W_G)=\cc(Q_G,W_G)$. 
\end{proof}

\subsection{Gorenstein affine toric threefolds}

Assume $k=\mathbb{C}$.

\begin{theorem}\label{thm:singularity-category-of-toric-threefold}
Let $X=\mathrm{Spec}R$ be a Gorenstein affine toric threefold. Then there is a quiver with potential $(Q',W')$ such that there is a triangle equivalence $\cd_{sg}(R)\simeq \cc^s(Q',W')$ up to direct summands. If $\cd_{sg}(R)$ is idempotent complete, then this is a triangle equivalence.
\end{theorem}
\begin{proof}
Going through the proofs of \cite[Theorems 8.6 and 8.5]{Broomhead12}, we see that there is a quiver with potential $(Q,W)$ such that
\begin{itemize}
\item[-] $J(Q,W)$ is bimodule $3$-Calabi--Yau,
\item[-] for any vertex $i$ of $Q$, there is an isomorphism $e_i J(Q,W)e_i\cong R$, and the $R$-module $e_iJ(Q,W)$ is reflexive such that $J(Q,W)\cong \End_R(e_iJ(Q,W))$.
\end{itemize}
By \cite[Lemma 2.11]{Broomhead12}, there is a grading on $Q$ making $(Q,W)$ a positively graded quiver with potential. So by Theorem~\ref{thm:Jacobian-CY==>Ginzburg-formal}, the projection $\Gamma(Q,W)\to J(Q,W)$ is a quasi-isomorphism of dg algebras. Fix an arbitrary vertex $i$ of $Q$ and let $(Q',W')$ be the quiver with potential obtained from $(Q,W)$ by deleting $i$. The rest of the proof is the same as that of Theorem~\ref{thm:singularity-category-of-quotient-singularity}.
\end{proof}

For certain Gorenstein affine toric threefold $X=\mathrm{Spec}R$ with isolated singularities, it is shown in \cite[Theorem 6.3]{AmiotIyamaReiten15} that $\cd_{sg}(R)$ is triangle equivalent to the generalised 2-cluster category $\cc_2(\underline{A})$ of some finite-dimensional algebra $\underline{A}$ of global dimension at most $2$. By \cite[Theorem 6.12(a)]{Keller11}, there is a quiver with potential $(Q',W')$ such that $\cc_2(\underline{A})$, and hence $\cd_{sg}(R)$, is triangle equivalent to $\cc(Q',W')$. As a consequence of Theorem~\ref{thm:singularity-category-of-toric-threefold}, we obtain such an equivalence for all Gorenstein affine toric threefolds with isolated singularity.
\begin{corollary}\label{cor:singularity-category-for-toric-threefolds-with-isolated-singularity}
Let $X=\mathrm{Spec}R$ be a Gorenstein affine toric threefold with isolated singularity. Then there is a quiver with potential $(Q',W')$ such that there is a triangle equivalence $\cd_{sg}(R)\simeq \cc^s(Q',W')$.
\end{corollary}
\begin{proof}
This follows immediately from Theorem~\ref{thm:singularity-category-of-toric-threefold} because $\cd_{sg}(R)$ is idempotent complete.
\end{proof}

We remark that $J(Q',W')$ is finite-dimensional under the conditions of \cite[Theorem 6.3]{AmiotIyamaReiten15} ($J(Q',W')$ is $B/\langle e_i\rangle$ there), so $\cc^s(Q',W')=\cc(Q',W')$. We believe that this is still true in the more general situation of Corollary~\ref{cor:singularity-category-for-toric-threefolds-with-isolated-singularity}, then Corollary~\ref{cor:singularity-category-for-isolated-quotient-singularities} would be a corollary of Corollary~\ref{cor:singularity-category-for-toric-threefolds-with-isolated-singularity} because all Gorenstein isolated 3D quotient singularities are toric (\cite[Theorem 1.1]{KuranoNishi12}).
If $e_iJ(Q,W)$ is maximal Cohen--Macaulay over $R$, then we can proceed as in the proof of Corollary~\ref{cor:singularity-category-for-isolated-quotient-singularities}.


\def\cprime{$'$}
\providecommand{\bysame}{\leavevmode\hbox to3em{\hrulefill}\thinspace}
\providecommand{\MR}{\relax\ifhmode\unskip\space\fi MR }
\providecommand{\MRhref}[2]{%
  \href{http://www.ams.org/mathscinet-getitem?mr=#1}{#2}
}
\providecommand{\href}[2]{#2}

\end{document}